\documentclass[11pt,reqno]{amsart}
\usepackage{amssymb,amsmath,amstext,amsthm,amsfonts}
\usepackage{euscript,graphicx,a4wide}

\usepackage{mathpazo} 

\usepackage[colorlinks=true,backref]{hyperref}

\numberwithin{equation}{section}

\theoremstyle{plain}
\newtheorem{maintheorem}{Theorem}

\newtheorem{maincorollary}[maintheorem]{Corollary}

\newtheorem{theorem}{Theorem}[section]
\newtheorem{proposition}[theorem]{Proposition}

\newtheorem{lemma}[theorem]{Lemma}
\newtheorem{claim}[theorem]{Claim}
\newtheorem{remark}[theorem]{Remark}

\theoremstyle{definition}
\newtheorem{definition}{Definition}
\newtheorem{conjecture}{Conjecture}
\newcommand{\dem}{\begin{proof}}
\newcommand{\cqd}{\end{proof}}

\renewcommand{\angle}[1]{\sphericalangle(#1)}

\newcommand{\RR}{{\mathbb R}}

\newcommand{\CC}{{\mathbb C}}

\newcommand{\ZZ}{{\mathbb Z}}

\newcommand{\sS}{{\mathbb S}}

\newcommand{\D}{\EuScript{D}}

\newcommand{\cO}{\EuScript{O}}

\newcommand{\V}{\EuScript{V}}
\newcommand{\U}{\EuScript{U}}
\newcommand{\G}{\EuScript{G}}
\newcommand{\X}{\EuScript{X}}
\newcommand{\cC}{\EuScript{C}}

\newcommand{\C}{{\mathcal C}}

\newcommand{\cU}{{\mathcal U}}

\newcommand{\cW}{{\mathcal W}}

\newcommand{\ov}{\overline}

\newcommand{\wt}{\widetilde}
\newcommand{\wh}{\widehat}

\renewcommand{\epsilon}{\varepsilon}

\newcommand{\qand}{\quad\text{and}\quad}

\DeclareMathOperator{\inte}{int}
\DeclareMathOperator{\diam}{diam}
\DeclareMathOperator{\spec}{sp}
\DeclareMathOperator{\dist}{dist}

\DeclareMathOperator{\per}{Per}
\DeclareMathOperator{\sing}{Sing}
\DeclareMathOperator{\crit}{Crit}
\DeclareMathOperator{\lip}{Lip}
\DeclareMathOperator{\close}{Closure}

\DeclareMathOperator{\indx}{Ind}

\DeclareMathOperator{\diver}{div}
\DeclareMathOperator{\sv}{sv}

\newcommand{\m}{{\rm Leb}}

\title[Sectional-hyperbolic expansiveness]{On robust expansiveness for
  sectional hyperbolic attracting sets}

\date{\today}

\author[V. Araujo]{Vitor Araujo}
\email{vitor.araujo.im.ufba@gmail.com, vitor.d.araujo@ufba.br}
\urladdr{https://sites.google.com/site/vdaraujo99/}
\address{Instituto de Matem\'atica e Estat\'{\i}stica,
  Universidade Federal da Bahia, Av. Ademar de Barros s/n,
  40170-110 Salvador, Brazil.}

\author[J. Cerqueira]{Junilson Cerqueira}
\email{junilson.cerqueira@ifba.edu.br, junilson\_cerqueira@hotmail.com}
\address{Instituto Federal de Educa\c{c}\~{a}o Ci\^{e}ncia e Tecnologia da Bahia-
  IFBA/Euclides da Cunha, Av. Apromiano Campus, n 900,BR 116 - Km 220
  48500-000 Euclides da Cunha, Brazil.}

\thanks{V.A. was partially supported by CNPq (Grant
  301392/2015-3) and FAPESB-Brazil; and J.C. was partially
  supported by CAPES-Brazil and FAPESB-Bahia, Brazil.}

\keywords{sectional-hyperbolicity, robust expansiveness,
  strong dissipativity, star flow, robust transitivity,
  robust chaotic, attracting sets}

\subjclass[2010]{Primary: 37C10. Secondary: 37D30, 37D50.}

\begin{document}

\begin{abstract} 
  We prove that sectional-hyperbolic attracting sets for $C^1$ vector
  fields are robustly expansive (under an open technical condition of
  strong dissipativeness for higher codimensional cases). This extends
  known results of expansiveness for singular-hyperbolic attractors in
  $3$-flows even in this low dimensional setting. We deduce a converse
  result taking advantage of recent progress in the study of star
  vector fields: a robustly transitive attractor is
  sectional-hyperbolic if, and only if, it is robustly expansive. In a
  low dimensional setting, we show that an attracting set of a
  $3$-flow is singular-hyperbolic if, and only if, it is robustly
  chaotic (robustly sensitive to initial conditions).
\end{abstract}

\maketitle

\tableofcontents

\section{Introduction}
\label{sec:introduction}

The theory of uniformly hyperbolic dynamics was initiated in the 1960s
by Smale \cite{Sm67} and, through the work of his students and
collaborators, as well as mathematicians in the Russian school
(e.g.~\cite{AP37,An67}), led to a great development of the field of
dynamical systems.  This elegant theory did not cover which turned out
to be important classes of dynamical systems: the most influential
examples being, arguably, the H{\'e}non map \cite{He76}, for the
discrete time case; and the Lorenz flow \cite{Lo63}, for the
continuous time case.

To extend the notion of uniform hyperbolicity to encompass
sets containing equilibria accumulated by recurrent orbits,
a fundamental step was given by Morales, Pacifico, and
Pujals in \cite{MPP98,MPP04}. There they proved that a
robustly transitive invariant attractor of a $3$-dimensional
flow that contains some equilibria must be singular
hyperbolic, i.e., it admits an invariant splitting
$E^s\oplus E^{cu}$ of the tangent bundle into a
$1$-dimensional uniformly contracting sub-bundle and a
$2$-dimensional volume-expanding sub-bundle.  

The first examples of singular hyperbolic sets included the
Lorenz attractor \cite{Lo63,LY85} and its geometric models
\cite{Gu76, ABS77, GW79,Wil79}, and the singular-horseshoe
\cite{LP86}, besides the uniformly hyperbolic sets
themselves. Many other examples have been found
e.g. \cite{MPS05, MPu97, MPP00,Morales07}. For arbitrary
dimensions this notion was extended first in~\cite{MeMor08}
by Metzger and Morales, and the first concrete example
provided by Bonatti, Pumari\~no and Viana in~\cite{BPV97}.
These are \emph{sectional-hyperbolic} attractors, where now
the splitting $E^s\oplus E^{cu}$ of the tangent bundle can
have $d_{cu}=\dim E^{cu}\ge2$, $d_s=\dim E^s\ge1$ and the
area along any $2$-subspace of $E^{cu}$ is uniformly
expanded by the tangent map of the flow.

In the absence of equilibria, both singular-hyperbolic sets
and sectional-hyperbolic sets are uniform hyperbolic. It is
natural to try to understand the dynamical consequences of
sectional hyperbolicity.

In \cite{APPV} the authors prove that all
singular-hyperbolic \emph{attractors} are \emph{expansive},
meaning, roughly, that any pair of orbits which remain close
at all times must actually coincide. There are different
notions of expansiveness and similar, as ''kinematic
expansive'' which are considered in \cite{BoWa72} and
\cite{KS79} and explored in \cite{CerLew10} and
\cite{artigue_2016}; and also ``rescaled expansivity''
\cite{Wen2018}. Here we focus on the one introduced by
Komuro \cite{Km84} to be compatible with the dynamics of the
geometric Lorenz attractor.

Here, building on the work \cite{APPV} and more recently
\cite{ArMel17, ArMel18, ArSzTr}, we extend the expansiveness property
obtained in \cite{APPV} from sectional-hyperbolic \emph{attractors} to
sectional-hyperbolic \emph{attracting sets}, extending the previous
result even in the $3$-dimensional case, avoiding the assumption of
the existence of a dense regular orbit in the set.

Moreover, we show that sectional-hyperbolic attracting sets are $C^1$
\emph{robustly expansive}: we can find uniform bounds on the distance
between pairs of orbits for any given $C^1$ nearby vector field so
that the orbits must coincide.  When $d_{cu}>2$ we need to assume a
\emph{strong dissipativity} condition on the vector field in a
neighborhood of the attracting set, which is still a $C^1$ open
condition.

The main tool of the proof is the construction of a global
Poincar\'{e} return map to a suitably chosen family of
cross-sections of the flow near the attracting set,
extending the constructions from \cite{APPV, ArMel18,
  ArSzTr} to any dimension $d_{cu}>2$.

We then explore some consequences of our results. It is
well-known that expansiveness implies $h$-expansiveness
(entropy expansiveness) \cite{Bowen72} and then the
semicontinuity of the entropy function ensuring the
existence of equilibrium states for all continuous
potentials. Recently \cite{PaYaYa} entropy-expansiveness has
been proved more directly for sectional-hyperbolic sets and
used to obtain several ergodic theoretical results.

In addition, robust expansiveness implies that the vector
field is a \emph{star vector field}, and this class has many
important features. Building on recent work from Wen
\cite{GW2006} together with Shi and Gan \cite{SGW14} we
obtain partial converses of the main result. Namely, \emph{a
  robustly expansive non-singular vector field is uniformly
  hyperbolic}; and \emph{a robustly transitive attractor is
  sectional-hyperbolic if, and only if, it is robustly
  expansive}.

Moreover, for $3$-flows we equate robust expansiveness and robust
sensitivity to initial conditions, which we denominate \emph{chaotic}
behavior, to obtain that \emph{an attracting set of a $3$-flow is
  singular-hyperbolic if, and only if, it is robustly chaotic}.
Robust expansivity is also useful to obtain stability of asymptotic
sojourn times given by physical measures, as recently explored in
\cite{araJSP21}.


\section{Statement of the results}
\label{sec:statement-results}

Let $M$ be a compact connected Riemannian manifold with dimension
$\dim M=m$, induced distance $d$ and volume form $\m$. Let $\X^r(M)$,
$r\ge1$, be the set of $C^r$ vector fields on $M$ and denote by
$\phi_t$ the flow generated by $G\in\X^r(M)$.

\subsection{Sectional-hyperbolic attracting sets}
\label{sec:PH}

An \emph{invariant set} $\Lambda$ for the flow $\phi_t$ is a
subset of $M$ which satisfies $\phi_t(\Lambda)=\Lambda$ for
all $t\in\RR$.  Given a compact invariant set $\Lambda$ for
$G\in \X^r(M)$, we say that $\Lambda$ is \emph{isolated} if
there exists an open set $U\supset \Lambda$ such that
$ \Lambda =\bigcap_{t\in\RR}\close{\phi_t(U)}$.  If $U$ can
be chosen so that $\close{\phi_t(U)}\subset U$ for all
$t>0$, then we say that $\Lambda$ is an \emph{attracting
  set} and $U$ a \emph{trapping region} (or \emph{isolated
  neighborhood}) for
$\Lambda=\Lambda_G(U)=\cap_{t>0}\close{\phi_t(U)}$.

For a compact invariant set $\Lambda$, we say that $\Lambda$ is {\em
  partially hyperbolic} if the tangent bundle over $\Lambda$ can be
written as a continuous $D\phi_t$-invariant sum
$ T_\Lambda M=E^s\oplus E^{cu}, $ where $d_s=\dim E^s_x\ge1$ and
$d_{cu}=\dim E^{cu}_x\ge2$ for $x\in\Lambda$, and there exist
constants $C>0$, $\lambda\in(0,1)$ such that for all $x \in \Lambda$,
$t\ge0$, we have
\begin{itemize}
\item \emph{uniform contraction along} $E^s$:
  $ \|D\phi_t | E^s_x\| \le C \lambda^t; $ and
\item \emph{domination of the splitting}:
  $ \|D\phi_t | E^s_x\| \cdot \|D\phi_{-t} | E^{cu}_{\phi_tx}\|
  \le C \lambda^t.  $
\end{itemize}
We say that $E^s$ is the \emph{stable bundle} and $E^{cu}$
the \emph{center-unstable bundle}.  A {\em partially
  hyperbolic attracting set} is a partially hyperbolic set
that is also an attracting set.

We say that the center-unstable bundle $E^{cu}$ is
\emph{volume expanding} if there exists $K,\theta>0$ such
that $|\det(D\phi_t| E^{cu}_x)|\geq K e^{\theta t}$ for all
$x\in \Lambda$, $t\geq 0$. More generally, $E^{cu}$ is {\em
  sectional expanding} if for every two-dimensional subspace
$P_x\subset E^{cu}_x$,
\begin{align} \label{eq:sectional} |\det(D\phi_t(x)\mid P_x
  )| \ge K e^{\theta t}\quad\text{for all $x \in \Lambda$,
    $t\ge0$}.
\end{align}

If $\sigma\in M$ and $G(\sigma)=0$, then $\sigma$ is called an {\em
  equilibrium} or \emph{singularity} in what follows and we denote by
$\sing(G)$ the family of all such points. We say that a singularity
$\sigma\in\sing(G)$ is \emph{hyperbolic} if all the eigenvalues of
$DG(\sigma)$ have non-zero real part.

A point $p\in M$ is \emph{periodic} for the flow $\phi_t$ generated by
$G$ if $G(p)\neq\vec0$ and there exists $\tau>0$ so that
$\phi_\tau(p)=p$; its orbit
$\cO_G(p)=\phi_{\RR}(p)=\phi_{[0,\tau]}(p)$ is a \emph{periodic
  orbit}, an invariant simple closed curve for the flow. The family of
periodic orbits of $G$ is written $\per(G)$.

The \emph{critical elements} $\crit(G)$ of a vector field $G$ are its
equilibria and periodic orbits, that is,
$\crit(G)=\sing(G)\cup\per(G)$.  An invariant set is \emph{nontrivial}
if it is not a critical element of the vector field.

We say that a compact invariant set $\Lambda$ is a
\emph{sectional hyperbolic set} if $\Lambda$ is partially
hyperbolic with sectional expanding center-unstable bundle
and all equilibria in $\Lambda$ are hyperbolic.  A singular
hyperbolic set which is also an attracting set is called a
{\em sectional hyperbolic attracting set}.

A \emph{singular hyperbolic set} is a compact invariant set
$\Lambda$ which is partially hyperbolic with volume
expanding center-unstable subbundle and all equilibria
within the set are hyperbolic. A sectional hyperbolic set is
singular hyperbolic and both notions coincide if, and only
if, $d_{cu}=2$.

A sectional hyperbolic set with no equilibria is necessarily a
\emph{hyperbolic set}, that is, the central unstable subbundle admits
a splitting $E^{cu}_x=\RR\{G(x)\}\oplus E^u_x$ for all $x\in\Lambda$
where $E^u_x$ is uniformly contracting under the time reversed flow;
see e.g.~\cite{AraPac2010}. That is, $\Lambda$ is a \emph{hyperbolic
  set if by definition} $T_\Lambda M=E^s\oplus\RR\{G\}\oplus E^u$.

A \emph{periodic orbit $\cO_G(p)$ is hyperbolic} if $\cO_G(p)$ is a
hyperbolic subset for $G$. If moreover $E^u$ is trivial (i.e.
$E^u_q=\{\vec0\}, q\in\cO_G(p)$), then the periodic orbit is a
\emph{periodic sink}.

A singular hyperbolic attracting set cannot contain isolated periodic
orbits.  For otherwise such orbit must be a periodic sink,
contradicting volume expansion.

We recall that a subset $\Lambda \subset M$ is
\emph{transitive} if it has a full dense orbit, that is,
there exists $x\in \Lambda$ such that
$\close{\{\phi_tx:t\ge0\}}=\Lambda=
\close{\{\phi_tx:t\le0\}}$.

A nontrivial transitive sectional hyperbolic attracting set is a
\emph{sectional hyperbolic attractor}.

The prototype of a sectional-hyperbolic attractor for
$3$-flows is the Lorenz attractor; see
e.g. \cite{Lo63,Tu99,AraPac2010}. For higher dimensional
flows we have the multidimensional Lorenz attractor: see
\cite{BPV97}. More examples are indicated in
Remarks~\ref{rmk:ds1} and~\ref{rmk:notLorenzlike} and many
more in\cite{Morales07}.

\subsection{Robust expansiveness for codimension-two
  sectional-hyperbolic attracting sets}

The flow is \emph{sensitive to initial conditions} if there
is $\delta>0$ such that, for any $x\in M$ and any
neighborhood $N$ of $x$, there is $y\in N$ and $t\in\RR$
such that $d(\phi_t(x), \phi_t(y))>\delta$.  We shall work
with a much stronger property.

\begin{definition}\label{def:expansiveness}
  Denote by $S(\RR)$ the set of surjective increasing continuous
  functions $h:\RR\to\RR$. We say that the flow is \emph{expansive} if
  for every $\epsilon>0$ there is $\delta>0$ such that, for any
  $h\in S(\RR)$
  \begin{align*}
    d(\phi_t(x),\phi_{h(t)}(y))\leq\delta,\quad \forall t\in\RR
    \implies
    \exists t_0\in\RR \text{ such that }
    \phi_{h(t_0)}(y)\in \phi_{[t_0-\epsilon,t_0+\epsilon]}(x).
  \end{align*}
We say that a invariant compact set $\Lambda$ is expansive if the
restriction of $\phi_t$ to $\Lambda$ is an expansive flow.
\end{definition}

This notion was proposed by Komuro in \cite{Ko84}.  We
consider a robust version.

\begin{definition}\label{def:robexpansiveness}
  We say that the vector field $G$ is \emph{robustly expansive} on an
  attracting set $\Lambda=\cap_{t>0}\ov{\phi_t(U})$ if there exists a
  neighborhood $\V$ of $G$ in $\X^1(M)$ such that for every
  $\epsilon>0$ there is $\delta>0$ such that, for any
  $x,y\in\Lambda_Y=\Lambda_Y(U)=\cap_{t>0}\ov{\psi_t(U)}$,
  $h\in S(\RR)$ and $Y\in\V$
  \begin{align*}
    d(\psi_t(x),\psi_{h(t)}(y))\leq\delta,\quad \forall t\in\RR
    \implies
    \exists t_0\in\RR \text{ such that }
    \psi_{h(t_0)}(y)\in \psi_{[t_0-\epsilon,t_0+\epsilon]}(x)
  \end{align*}
where $\psi_t$ is the flow generated by $Y$.
\end{definition}

Our results show that a sectional hyperbolic attracting set $\Lambda$
is robustly expansive.

\begin{maintheorem}\label{mthm:principal1}
Every sectional hyperbolic attracting set of a vector field
$G\in\EuScript{X}^1(M)$, with $d_{cu}=2$, is $C^1$ robustly expansive.
\end{maintheorem}

\subsection{Robust expansiveness for higher codimension}
\label{sec:robust-expans-higher}

In the higher condimension case $d_{cu}>2$, we need to assume that
$\Lambda$ satisfies a ``strongly dissipative'' condition, which is
equivalent to a bunching condition on the partially hyperbolic
splitting, but simpler to check for flows induced by vector fields.
We implicitly assume without loss of generality that the compact
manifold $M$ is embedded in an Euclidian space to simplifiy the
statement of this condition.

\begin{definition}\label{def:qstrongdiss}
  Let us fix $q>1/d_s$. We say that a partially hyperbolic attracting
  set $\Lambda$ is \emph{$q$-strongly dissipative} if
  \begin{itemize}
  \item[(a)] for every equilibria $\sigma\in\Lambda$ (if any), the
    eigenvalues $\lambda_j$ of $DG(\sigma)$, ordered so
    that\footnote{Here $\Re s$ denotes the real part of $z\in\CC$.}
    $\Re\lambda_1\leq\Re\lambda_2\leq\cdots\le \Re\lambda_d$, satisfy
    $\Re(\lambda_1-\lambda_{d_s+1}+q\lambda_d)<0$;
  \item[(b)]
    $\sup_{x\in\Lambda}\{\diver G(x)+(d_sq-1)\|(DG)(x)\|_2\}<0$, where
    $\|\cdot\|_2$ denotes the matricial norm given by
    $\|A\|_2=(\Sigma_{ij}a^2_{ij})^\frac{1}{2}$ for a matrix
    $A_{d\times d}$.
  \end{itemize}
\end{definition}
This condition was introduced in~\cite{ArMel17} where it was shown to
imply that the stable foliation associated to the partial hyperbolic
attracting set extends to a $C^1$-smooth topological foliation of the
basin of attraction of $\Lambda$.

\begin{maintheorem}\label{mthm:principal2}
  Every sectional hyperbolic attracting set of $1$-strongly
  dissipative vector field $G\in\X^1(M)$ is $C^1$ robustly
  expansive.
\end{maintheorem}

\begin{remark}
  \label{rmk:ds1}
  The multidimensional Lorenz class of examples introduced in
  \cite{BPV97} provides classes of sectional-hyperbolic attractors for
  each choice of $d_{cu}\ge2$ and $d_s\ge2$; and also an example with
  $d_{cu}=3$ and $d_s=1$. There are plenty of singular-hyperbolic
  examples: see e.g.~\cite{Morales07} and references therein.

  In the the cases $d_s=1$ with $d_{cu}>2$ the $1$-strongly
  dissipative assumption is interpreted to mean ``$q$-strongly
  dissipative for some $q>1$''.
\end{remark}

Since we need the strong dissipativeness condition for
technical reasons, we naturally pose the following.

\begin{conjecture}\label{conj:nostrongdiss}
  Theorem~\ref{mthm:principal2} is still true for all $C^1$
  vector fields exhibiting a sectional hyperbolic attracting set.
\end{conjecture}

\subsection{Some consequences of robust expansiveness}
\label{sec:some-conseq-robust}

We in fact obtain a slightly stronger result: the main
argument provides a proof of (robust) \emph{positive
  expansiveness}, that is, sectional-hyperbolic attracting
sets in the setting of Theorems~\ref{mthm:principal1}
and~\ref{mthm:principal2} satisfy: for each $\epsilon>0$ we
can find $\delta>0$ so that
\begin{align*}
  d(\phi_t(x),\phi_{h(t)}(y))\leq\delta,
  h\in S(\RR),\forall t>0
    \implies
    \exists t_0>0, s\in(-\epsilon,\epsilon):
    \phi_{h(t_0)}(y)\in \cW^s_{\phi_{t_0+s}(x)};
\end{align*}
where $\cW^s_z$ is the local stable manifold through points
$z\in U_0$, which are well-defined for partially hyperbolic
attracting sets; see
Section~\ref{sec:preliminary-results}. We note that a
slightly stronger notion of positive expansiveness (akin to
Bowen-Walters expansiveness) has been shown in
\cite{artigue_2014} to imply finitely many periodic orbits
only.

From positive expansiveness, provided by
Theorem~\ref{thm:expansivepoincare}, robust expansiveness
follows as explained in Section~\ref{sec:conclus-proof},
exploring the properties of stable manifolds of partially
hyperbolic sets.

This enable us to obtain partial converses to the statements
of the main Theorems~\ref{mthm:principal1}
and~\ref{mthm:principal2} extending the results
of~\cite{MPP99} by (roughly) reinforcing robust transitivity
with robust expansiveness.

We need the following standard notion. Given $G\in\X^1(M)$ and
$x\in M$ we denote the \emph{omega-limit} set
\begin{align*}
  \omega(x)=\omega_G(x)=\left\{y\in M: \exists t_n\nearrow\infty\text{  s.t.
  } \phi_{t_n}x\xrightarrow[n\to\infty]{}y \right\}
\end{align*}
and the \emph{alpha-limit} set $\alpha(x)=\omega_{-G}(x)$, which are
both non-empty on a compact ambient space $M$.

\subsubsection{Robust expansive flows are star flows}
\label{sec:robust-expans-flows}

A vector field $G\in \X^1(M)$ is a {\em star vector field} if there
exists a $C^1$ neighborhood $\U$ of $G$ such that every critical
element of every $Y\in \U$ is hyperbolic. The set of $C^1$ star vector
fields of $M$ is denoted by $\X^*(M)$. The following is an extension
of \cite{MorSakSun05} to robust expansive flows in the sense of Komuro
which encompasses singular flows; see e.g. \cite{Km84,APPV}. A proof
of this can be found e.g. in \cite[Theorem A]{teseSenos}.

\begin{theorem}\label{thm:star}
  A robustly expansive vector field $G\in\X^1(M)$ is a star vector
  field: $G\in\X^*(M)$.

  Moreover, if $G\in\X^1(M)$ is robustly expansive on the attracting
  set $\Lambda_X(U)$ with trapping region $U$, then $G$ is a star
  vector field in $U$: there exists a neighborhood $\U\subset\X^1(M)$
  of $G$ such that all critical elements of each $Y\in\U$ contained in
  $U$ are hyperbolic.
\end{theorem}

This is a very strong condition for non-singular vector fields:
putting the last result together with \cite{GW2006} we obtain that
every robustly expansive non-singular vector field $G$ is an Axiom A
vector field satisfying the no-cycles condition.



\subsubsection{Robustly transitive and expansive attractors are
  sectional-hyperbolic}
\label{sec:robustly-transit-exp}

Using the recent developments in the study of singular star flows
from~\cite{SGW14} we are able to prove the following. We say that an
attractor $\Lambda_G(U_0)$ is \emph{robustly transitive} if there
exists a $C^1$ neighborhood $\U\subset\X^1(M)$ such that
$\Lambda_Y(U_0)$ is transitive for all $Y\in\U$.

\begin{maincorollary}\label{mcor:starattractsechyp}
  Every robustly transitive and expansive attractor of $G\in\X^1(M)$
  is sectional-hyperbolic.

  In particular, every robustly transitive attractor of a star vector
  field is sectional-hyperbolic.
\end{maincorollary}

\begin{remark}\label{rmk:robtransnotenough}
  Although robust transitivity of an attracting set alone
  implies sectional-hyperbolicity for $3$-flows, which is
  the main result of~\cite{MPP04}, this is not enough to
  ensure sectional-hyperbolicity for $m$-dimensional flows
  with $m\ge4$. Indeed, the ``wild strange attractor''
  presented by Shilnikov and Turaev in \cite{ST98} is a
  robustly transitive singular attractor which is not
  sectional hyperbolic; see e.g. \cite[Corollary
  4.3]{MeMor08} and \cite[Lemma 4.2 and Theorem 4]{ST98}.
\end{remark}

\begin{remark}
  \label{rmk:ergodic}
  Sectional-hyperbolicity for attracting sets of vector
  fields of class $C^{1+\epsilon}$, for any given fixed
  $\epsilon>0$, implies some strong ergodic properties:
  existence of physical/SRB measures whose basins cover a
  full Lebesgue measure subset of the trapping region
  \cite{APPV,ArSzTr,LeplYa17,araujo_2021}; and these measures
  are rapid mixing and satisfy the Almost Sure Invariance
  Principle (which implies many other statistical
  properties: Central Limit Theorem, Law of the Iterated
  Logarithm etc) on an open and dense subset of such vector
  fields \cite{ArMel18}. Since $\X^2(M)$ is $C^1$ dense in
  $\X^1(M)$ with the $C^1$ topology, all these results hold
  for a dense subset of the $C^1$ open classes and,
  naturally, also for $C^2$ open classes of such vector
  fields in Corollary~\ref{mcor:starattractsechyp}.
\end{remark}

Coupling Corollary~\ref{mcor:starattractsechyp} with
Theorem~\ref{mthm:principal2} we obtain a partial converse
to this theorem.

\begin{maincorollary}\label{mcor:sechypiffrobexp}
  If $G\in\X^1(M)$ admits a robustly transitive attractor
  $\Lambda=\Lambda_G(U_0)$ with trapping region $U_0$ where $G$ is
  $1$-strongly dissipative, then
  \begin{align*}
    \Lambda \text{  is sectional-hyperbolic }
    \iff
    \Lambda \text{ is robustly expansive.}
  \end{align*}
\end{maincorollary}

In~\cite{SGW14} the authors show that $C^1$ generically
among star vector fields $G$ on $4$-manifolds every Lyapunov
stable chain recurrence class is sectional-hyperbolic,
either for $G$ or for $-G$.

If $\dim M\ge5$ there are transitive star flows with singularities of
different indices \cite{daLuz18}.  For higher dimensional results we
may additionally assume homogeneity; see the next subsection.

For conservative flows it is known that $C^1$ stably expansive
conservative flows are Anosov, that is, globally hyperbolic; see
e.g. \cite{BesLeeWen}.

\subsubsection{Robust chaoticity and sectional-hyperbolicity}
\label{sec:robust-expans-robust}

Recall that an invariant attracting subset $\Lambda$ is
\emph{sensitive to initial conditions} if, for every small enough
$r>0$ and $x\in\Lambda$, and for any neighborhood $U$ of $x$, there
exists $y\in U$ and $t\neq0$ such that $\phi_ty$ and $\phi_tx$ are
$r$-apart from each other: $d\big(\phi_ty,\phi_tx\big)\ge r$.

We say that an invariant subset $\Lambda$ for a flow
$\phi_t$ is \emph{future chaotic with constant $r>0$} if,
for every $x\in\Lambda$ and each neighborhood $U$ of $x$ in
the ambient manifold, there exists $y\in U$ and $t>0$ such
that $d\big(\phi_ty,\phi_tx\big)\ge r$.  Analogously, we say
that $\Lambda$ is \emph{past chaotic with constant $r$} if
$\Lambda$ is future chaotic with constant $r$ for the
reverse flow $\phi_{-t}$ (i.e., generated by $-G$).

If we have such \emph{sensitive dependence both for the past
  and for the future}, we say that $\Lambda$ is
\emph{chaotic}. Note that sensitive dependence on initial
conditions is weaker than chaotic, future chaotic or past
chaotic conditions.

Clearly, expansiveness implies sensitive dependence on
initial conditions. An argument with the same flavor as the
proof of expansiveness provides the following (see also
\cite{AMS2010} for a different approach to sensitiveness).

\begin{maincorollary}\label{mcor:sing-hyp-chaotic}
  A sectional-hyperbolic attracting set $\Lambda=\Lambda_G(U)$ is
  \emph{robustly chaotic}, i.e. there exists a neighborhood $\cU$ of
  $G$ in $\X^1(M)$ and a constant $r_0>0$ such that
  $\Lambda_Y(U)=\cap_{t>0}\close{\psi_t(U)}$ is chaotic with constant
  $r_0$ for each $Y\in\cU$, where $U$ is a trapping region for
  $\Lambda$ and $\psi_t$ is the flow generated by $Y$.
\end{maincorollary}

For a partially hyperbolic attracting set of codimension two we obtain
a converse to Theorem~\ref{mthm:principal1}.

\begin{maincorollary}\label{mcor:dcu2robchaotic}
  Let $\Lambda$ be a partially hyperbolic attracting set for
  $G\in\X^1(M)$ with $d_{cu}=2$. Then $\Lambda$ is
  sectional-hyperbolic if, and only if, $\Lambda$ is robustly chaotic.
\end{maincorollary}

In the three-dimensional case we obtain the converse to
Theorem~\ref{mthm:principal1}.

\begin{maincorollary}
  \label{mcor:robust-chaotic-sing-hyp}
  Let $\Lambda$ be an attracting set for $G\in\X^1(M^3)$. Then
  $\Lambda$ is sectional-hyperbolic if, and only if, $\Lambda$ is
  robustly chaotic.
\end{maincorollary}

Hence, robustly chaotic singular-attracting sets are necessarily
Lorenz-like. Thus, \emph{if we can show that arbitrarily close orbits,
  in a small neighborhood of an attracting set, are driven apart, for
  the future as well as for the past, by the evolution of the system,
  and this behavior persists for all $C^1$ nearby three-dimensional
  vector fields, then the attracting set is sectional-hyperbolic}.

We can extend this conclusion to higher dimensions assuming a stronger
condition.

We say that a singularity $\sigma$ is \emph{generalized Lorenz-like}
if $DG(\sigma)|E^{cu}_\sigma$ has a real eigenvalue $\lambda^s$ and
$\lambda^u=\inf\{\Re(\lambda):\lambda\in\spec(DG(\sigma)),
\Re(\lambda)\ge0\}$ satisfies $-\lambda^u<\lambda^s<0<\lambda^u$ (so
the index of $\sigma$ is $\dim E^s_\sigma=d_s+1$).  We say that $G$ is
a \emph{robustly homogeneous} vector field on the trapping region $U$
if, for some integer $1\le i+1<m$ and for each vector field $Y$ in a
$C^1$ neighborhood $\U$ of $G$:
\begin{itemize}
\item the singularities in $U$ are generalized Lorenz-like with index
  $i$ or $i+1$; and 
\item periodic orbits in $U$ are hyperbolic of saddle-type with the
  same index $i$.
\end{itemize}
Note that homogeneity is stronger than the star condition since the
latter admits the coexistence of critical elements with arbitrary
indices.

The following was already essentially obtained by Metzger and Morales
in~\cite{MeMor08}; see Section~\ref{sec:singhypattracting} for a
proof.

\begin{theorem}\label{thm:robusthomogeneous}
  Let $\Lambda=\Lambda_G(U)$ be an attracting set for
  $G\in\X^1(M)$. If $G$ is robustly homogeneous in $U$, then $\Lambda$
  is sectional-hyperbolic.
\end{theorem}

This result is also a tool needed to prove
Corollaries~\ref{mcor:dcu2robchaotic}
and~\ref{mcor:robust-chaotic-sing-hyp}.

\subsection{Organization of the text}
\label{sec:organization-text}

We present preliminary results on sectional-hyperbolic attracting sets
in Section~\ref{sec:preliminary-results} which are needed for the
proofs of Theorems~\ref{mthm:principal1} and~\ref{mthm:principal2}, to
be presented in Section~\ref{sec:proof-main-theorems}.


In Section~\ref{sec:singhypattracting} we present an overview of the
proof of Theorem~\ref{thm:robusthomogeneous} and, using this result as
a tool together with all the previous results, we prove
Corollary~\ref{mcor:starattractsechyp} in
Subsection~\ref{sec:robust-expans-attrac}; and
Corollaries~\ref{mcor:sing-hyp-chaotic}, \ref{mcor:dcu2robchaotic}
and~\ref{mcor:robust-chaotic-sing-hyp} in
Subsection~\ref{sec:robust-chaoticity}.

\subsection*{Acknowledgements}

This is based on the PhD thesis of J. Cerqueira at the Instituto de
Matematica e Estatistica-Universidade Federal da Bahia (UFBA) under a
CAPES/FAPESB scholarship. J. C.  thanks the Mathematics and Statistics
Institute at UFBA for the use of its facilities and the financial
support from CAPES and FAPESB during his M.Sc. and Ph.D. studies.  We
thank A. Castro; V. Pinheiro; L. Salgado and F. Santos for many
comments and suggestions which greatly improved the text.  The authors
also thank the anonymous referees for the careful reading and the many
useful suggestions that greatly helped to improve the quality of the
text.


\section{Preliminary results on sectional-hyperbolic attracting sets}
\label{sec:preliminary-results}

Let $G$ be a $C^1$ vector field admitting a
singular-hyperbolic attracting set $\Lambda$ with isolating
neighborhood $U$.

\subsection{Generalized Lorenz-like singularities}
\label{sec:lorenz-like-singul}

We recall some properties of sectional-hyperbolic
attracting sets extending some results from~\cite{ArMel17,
  ArMel18} which hold for $d_{cu}\ge2$.

\begin{proposition} \label{prop:generaLorenzlike} Let $\Lambda$ be a
  sectional hyperbolic attracting set and let $\sigma\in\Lambda$ be an
  equilibrium.  If there exists $x\in\Lambda\setminus\{\sigma\}$ so
  that $\sigma\in\omega(x)$, then $\sigma$ is generalized Lorenz-like.
\end{proposition}

\begin{remark}
  \label{rmk:notLorenzlike}
  \begin{enumerate}
  \item If $\sigma\in\sing(G)\cap\Lambda$ is generalized Lorenz-like,
    then at $w\in\gamma_\sigma^s\setminus\{\sigma\}$\footnote{An
      embedded disk $\gamma\subset M$ is a (local) {\em stable disk},
      if $d(\phi_t x, \phi_t y)\to0$ exponentially fast as
      $t\to+\infty$, for $x,y\in\gamma$. Here $\gamma^s_\sigma$
      is a local stable disk containing $\sigma$ with maximal
      dimension: the local stable manifold; see e.g. \cite{PM82}.} we
    have $T_w\gamma_\sigma^s= E^{cs}_w=E^s_w\oplus\RR\cdot\{G(w)\}$
    since $T\gamma_\sigma^s$ is $D\phi_t$-invariant and contains
    $G(w)$ (because $\gamma_\sigma^s$ is $\phi_t$-invariant) and the
    dimensions coincide.
  \item If an equilibrium $\sigma\in\sing(G)\cap\Lambda$ is not
    generalized Lorenz-like, then $\sigma$ is not in the positive
    limit set of $\Lambda\setminus\{\sigma\}$, i.e. there is no
    $x\in\Lambda\setminus\{\sigma\}$ so that
    $\sigma\in\omega(x)$. Moreover, we have
    $\dim E^s_\sigma\in\{d_s,d_s+1\}$. An example is provided by the
    pair of equilibria of the Lorenz system of equations away from the
    origin: these are saddles with an expanding complex eigenvalue
    which belong to the attracting set of the trapping ellipsoid
    already known to E. Lorenz; see e.g. \cite[Section
    3.3]{AraPac2010} and references therein.
  \item There are many examples of singular-hyperbolic attracting
    sets, non-transitive and containing non-Lorenz-like (generalized)
    singularities; see Figure~\ref{fig:singhypattracting} for an
    example obtained by conveniently modifying the geometric Lorenz
    construction, and many others in \cite{Morales07} or more recently
    in \cite{BaBoPac21}.
  \item In what follows, \emph{a singular-hyperbolic attracting set
      with no (generalized) Lorenz-like singularities can be treated
      as non-singular attracting set}, since non-Lorenz-like
    singularites do not interfere with the asymptotic dynamics of
    positive trajectories of points in the set.
  \item There are examples of three-dimensional singular
    \emph{attractors} with non-Lorenz-like singularities, but these
    sets are not robustly transitive and cannot be
    sectional-hyperbolic; see e.g.~\cite{MPu98}.
  \item A singular hyperbolic attracting set contains no isolated
    periodic orbits.  For such a periodic orbit would have to be a
    periodic sink, violating volume expansion.
  \end{enumerate}
\end{remark}

\begin{figure}[h]
\centering
\includegraphics[width=9cm]{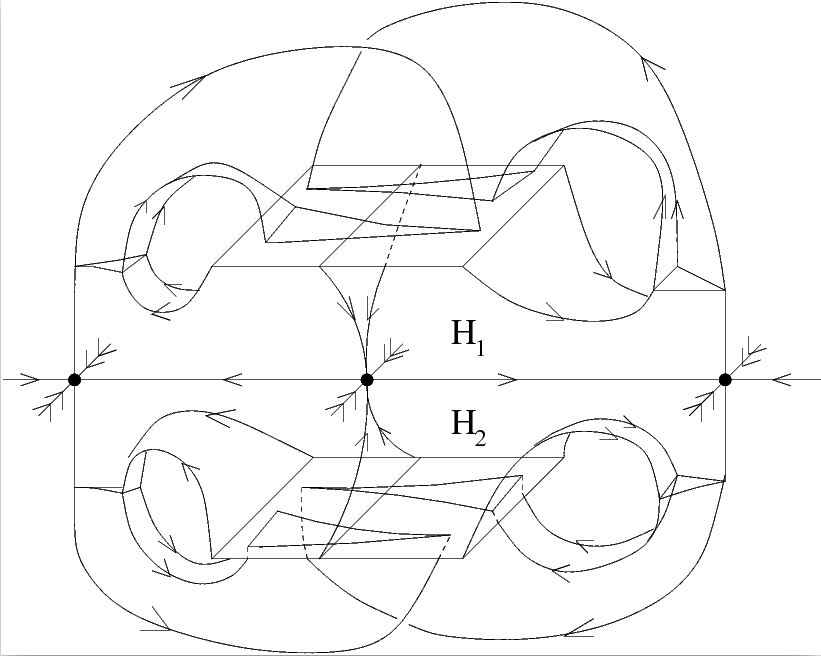}
\caption{\label{fig:singhypattracting}Example of a
  singular-hyperbolic attracting set, non-transitive (in
  fact, it is the union of two transitive sets indicated by
  $H_1,H_2$ above) and containing non-Lorenz like
  singularities.}
\end{figure}

\begin{proof}[Proof of Proposition~\ref{prop:generaLorenzlike}]
  It follows from sectional-hyperbolicity that $\sigma$ is a
  hyperbolic saddle and that at most $d_{cu}$ eigenvalues
  have positive real part.  If there are only $d_{cu}-1$ such
  eigenvalues, then the constraints on $\lambda^s$ and
  $\lambda^u$ follow from sectional expansion.

  Let $\gamma=\gamma_\sigma^s$ be the local stable manifold for
  $\sigma$.  If $\sigma\in\omega(x)$ for some
  $x\in\Lambda\setminus\{\sigma\}$, it remains to rule out the case
  $\dim \gamma=m-d_{cu}=d_s$.  In this case, $T_p\gamma=E^s_p$ for all
  $p\in \gamma\cap\Lambda$ and in particular $G(p)\in E^s_p$.

  On the one hand, $G(p)\in E^{cu}_p$ (see
  e.g.~\cite[Lemma~6.1]{AraPac2010}), so we deduce that $G(p)=0$ for
  all $p\in \gamma\cap\Lambda$ and so $\gamma\cap\Lambda=\{\sigma\}$.

  On the other hand, since $\sigma\in\omega(x)$,
  by the local behavior of orbits near hyperbolic saddles, there
  exists
  $p\in (\gamma\setminus\{\sigma\})\cap\omega(x)\subset
  (\gamma\setminus\{\sigma\})\cap\Lambda$ which, as we have seen, is
  impossible.
\end{proof}

\subsection{Invariant extension of the stable bundle}
\label{sec:invari-extens-stable}

Every partially hyperbolic attracting set admits an invariant
extension of the stable bundle, and also of the stable foliation, to
an open neighborhood, which we may assume without loss of generality
to be a trapping region $U_0$.

Let $\D^k$ denote the $k$-dimensional open unit disk and let
$\mathrm{Emb}^r(\D^k,M)$ denote the set of $C^r$ embeddings
$\psi:\D^k\to M$ endowed with the $C^r$ distance. We say
that \emph{the image of any such embedding is a $C^r$
  $k$-dimensional disk}.

\begin{theorem}
  \label{thm:Ws}
  Let $\Lambda$ be a partially hyperbolic attracting set.
  \begin{enumerate}
  \item The stable bundle $E^s$ over $\Lambda$ extends to a
    continuous uniformly contracting $D\phi_t$-invariant
    bundle $E^s$ on an open positively invariant
    neighborhood $U_0$ of $\Lambda$.
  \item There exists a constant $\lambda\in(0,1)$, such that
    \begin{enumerate}
    \item for every point $x \in U_0$ there is a $C^1$
      embedded $d_s$-dimensional disk $W^s_x\subset M$, with
      $x\in W^s_x$, such that $T_xW^s_x=E^s_x$;
      $\phi_t(W^s_x)\subset W^s_{\phi_tx}$ and
      $d(\phi_tx,\phi_ty)\le \lambda^t d(x,y)$ for all
      $y\in W^s_x$, $t\ge0$ and $n\ge1$.

    \item the disks $W^s_x$ depend continuously on $x$ in
      the $C^0$ topology: there is a continuous map\footnote{More
        precisely, $x\mapsto\gamma(x)(u)$ and $x\mapsto D\gamma(x)_u$
        are continuous maps, for each fixed $u\in\D^{d_s}$.}
      $\gamma:U_0\to {\rm Emb}^1(\D^{d_s},M)$ such that
      $\gamma(x)(0)=x$ and $\gamma(x)(\D^{d_s})=W^s_x$.
      Moreover, there exists $L>0$ such that
      $\lip\gamma(x)\le L$ for all $x\in U_0$.

    \item the family of disks $\{W^s_x:x\in U_0\}$ defines a
      topological foliation $\cW^s$ of $U_0$.
    \end{enumerate}
  \end{enumerate}
\end{theorem}

\begin{proof}
  This can be found in~\cite[Proposition~3.2, Theorem~4.2 and
  Lemma~4.8]{ArMel17}.
\end{proof}

\subsubsection{Smoothness of the stable foliation on a trapping region}
\label{sec:smoothn-topolog-foli}

For a sectional hyperbolic attracting set $\Lambda$, the trapping
region $U_0$ admits a $C^1$ topological foliation $\cW^s$ if we assume
that $\Lambda$ is $1$-strongly dissipative; recall
Definition~\ref{def:qstrongdiss}.

\begin{theorem}\label{thm:qdiss}
  Let $\Lambda$ be a sectional hyperbolic attracting set $\Lambda$ for
  a vector field $G$ of class $C^r$, for some $r\ge1$.  Suppose that
  $\Lambda$ is $q$-strongly dissipative for some $q\in
  [1/d_s,r]$. Then there exists a neighbourhood $U_0$ of $\Lambda$
  such that the stable manifolds $\{W^s_x:x\in U_0\}$ define a $C^q$
  foliation\footnote{Now we have that the map
    $(x,u)\mapsto \gamma(x)(u)$ is $C^q$.} of $U_0$.
\end{theorem}

\begin{proof}
  This is proved in \cite[Theorem 4.2]{ArMel17}.
\end{proof}

Theorem~\ref{thm:qdiss} with $q\ge1$ is crucial to have a good
geometrical estimate of distances between stable leaves close to the
attracting set in a higher codimension setting $d_{cu}>2$, as
explained in Subsection~\ref{sec:distance-between-sta} and used in
Section~\ref{sec:proof-main-theorems}.


\subsection{Extension of the center-unstable cone field}
\label{sec:extens-center-unstab}

The splitting $T_\Lambda M=E^s\oplus E^{cu}$ extends
continuously to a splitting $T_{U_0} M=E^s\oplus E^{cu}$
where $E^s$ is the invariant uniformly contracting bundle in
Theorem~\ref{thm:Ws} (however $E^{cu}$ is not invariant
in general).  Given $a>0$ and $x\in U_0$, we define the {\em
  center-unstable cone field} as
$
\cC^{cu}_x(a)=\{v= v^s+v^{cu}\in E^s_x\oplus E^{cu}_x:\|v^s\|\le a\|v^{cu}\|\}$.

\begin{proposition}
  \label{prop:Ccu}
  Let $\Lambda$ be a partially hyperbolic attracting set.
  \begin{enumerate}
  \item There exists $T_0>0$ such that for any $a>0$, after
    possibly shrinking $U_0$,
    $ D\phi_t\cdot \cC^{cu}_x(a)\subset
    \cC^{cu}_{\phi_tx}(a)$ for all $t\ge T_0$, $x\in U_0$.
  \item Let $\lambda_1\in(0,1)$ be given.  After possibly
    increasing $T_0$ and shrinking $U_0$, there exist
    constants $K,\theta>0$ such that
    $|\det(D\phi_t| P_x)|\geq K \, e^{\theta t}$ for each
    $2$-dimensional subspace $P_x\subset E^{cu}_x$ and all
    $x\in U_0$, $t\geq 0$.
  \end{enumerate}

\end{proposition}

\begin{proof} For item (1) see~\cite[Proposition~3.1]{ArMel17}.  Item
  (2) follows from the robustness of sectional expansion; see
  \cite[Proposition 2.10]{ArMel18} with straightforward adaptation to
  area expansion along any two-dimensional subspace of $E^{cu}_x$.
\end{proof}

\subsection{Global Poincar\'e map on adapted cross-sections}
\label{sec:global-poincare-map}

We assume that $\Lambda$ is a partially hyperbolic
attracting set and recall how to construct a piecewise
smooth Poincar\'e map $f:\Xi\to \Xi$ preserving a
contracting stable foliation $\cW^s(\Xi)$.  This largely
follows~\cite{APPV} (see also~\cite[Chapter~6]{AraPac2010})
and \cite[Section 3]{ArMel18} with slight modifications to
account for the higher dimensional set up.

We write $\rho_0>0$ for the injectivity radius of the
exponential map $\exp_z:T_zM\to M$ for all $z\in U$, so that
$\exp_z\mid B_z(0,\rho_0): B_z(0,\rho_0)\to M,
v\mapsto\exp_zv$ is a diffeomorphism with
$B_z(0,\rho_0)=\{v\in T_zM: \|v\|\le\rho_0\}$ and
$D\exp_z(0)=Id$ and also $d(z,\exp_z(v))=\|v\|$ for all
$v\in B_z(0,\rho_0)$.

\subsubsection{Construction of a global adapted cross-section}
\label{sec:constr-global-adapte}

Let $y\in\Lambda$ be a regular point ($G(y)\neq\vec0$).
Then there exists an open flow box $V_y\subset U_0$
containing $y$. That is, if we fix $\epsilon_0\in(0,1)$
small, then we can find a diffeomorphism
$\chi:\D^{d-1}\times(-\epsilon_0,\epsilon_0)\to V_y$ with
$\chi(0,0)=y$ such that
$\chi^{-1}\circ \phi_t\circ\chi(z,s)=(z,s+t)$.  Define the
cross-section $\Sigma_y=\chi(\D^{d-1}\times\{0\})$.


For each $x\in\Sigma_y$, let
$W^s_x(\Sigma_y)= \bigcup_{|t|<\epsilon_0}\phi_t(W^s_x)\cap
\Sigma_y$.  This defines a topological foliation
$\cW^s(\Sigma_y)$ of $\Sigma_y$.  We can also assume that
$\Sigma_y$ is diffeomorphic to $\D^{d_{cu}-1}\times\D^{d_s}$
by reducing the size of the $\Sigma_y$ if needed.  The
stable boundary
$\partial^s\Sigma_y\cong \partial\D^{d_{cu}-1}\times
\D^{d_s}\cong \sS^{d_{cu}-2}\times\D^{d_s}$ is a regular topological
manifold homeomorphic to a cylinder of stable leaves, since
$\cW^s$ is a topological foliation; i.e. $\cong$ denotes
only the existence of a homeomorphism and the subspace
topology of $\partial^s\Sigma_y$ induced by $M$ coincides
with the manifold topology.

Let $\D_a^{d_s}$ denote the \emph{closed disk} of radius $a\in(0,1]$
in $\RR^{d_s}$.  Define the {\em sub-cross-section}
$\Sigma_y(a)\cong \inte(\D^{d_{cu}-1}_a\times \D_a^{d_s})$, and the
corresponding sub-flow box
$V_y(a)\cong\inte(\Sigma_y(a))\times(-\epsilon_0,\epsilon_0)$
consisting of trajectories in $V_y$ which pass through
$\inte(\Sigma_y(a))\cong \inte(\D^{d_{cu}-1}_a)\times
\inte(\D_a^{d_s})$. In what follows we fix $a_0=3/4$.

For future reference, we also set
$\wh{\Sigma_y(a)}\cong \D^{d_{cu}-1}_a\times \D^{d_s}$ and the
corresponding sub-flow box
$\wh{V_y(a)}\cong\inte(\wh{\Sigma_y(a)})\times(-\epsilon_0,\epsilon_0)$;
see Figure~\ref{fig:subsection} for a sketch of
$\Sigma_y, \Sigma_y(a)$ and $\wh{\Sigma_y(a)}$.

\begin{figure}[h]
\centering
\includegraphics[width=10cm]{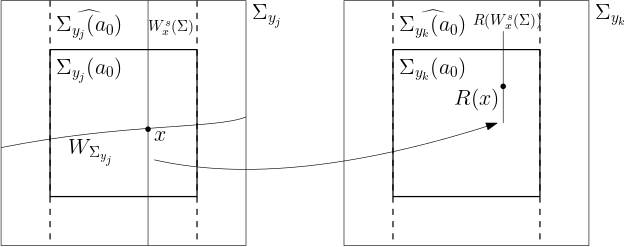}
\caption{\label{fig:subsection} Sketch of cross-sections and
  sub-cross-sections together with the crossing $cu$-disk; and the
  definition of the Poincar\'e time at a given point in the $cu$-disk
  and the corresponding Poincar\'e map on the stable leaf.}
\end{figure}

For each equilibrium $\sigma\in\Lambda$, we let $V_\sigma$
be an open neighborhood of $\sigma$ on which the flow is
linearizable.  Let $\gamma^s_\sigma$ and $\gamma^u_\sigma$
denote the local stable and unstable manifolds of $\sigma$
within $V_\sigma$; trajectories starting in $V_\sigma$
remain in $V_\sigma$ for all future time if and only if they
lie in $\gamma^s_\sigma$.


Define $V_0=\bigcup_{\sigma\in\sing(G)\cap U} V_\sigma$.  We shrink
the neighborhoods $V_\sigma$ so that they are disjoint;
$\Lambda\not\subset V_0$; and
$\gamma^u_\sigma\cap\partial V_\sigma\subset
\cup_{i=1}^{\ell_\sigma}V_{y_i}(a_0)$ for some regular points
$y_i=y_i(\sigma), i=1,\dots,\ell_\sigma$.

By compactness of $\Lambda$, there exists $\ell\in\ZZ^+$ and
regular points $y_1,\dots,y_\ell\in\Lambda$ such that
$\Lambda\setminus V_0 \subset \bigcup_{j=1}^\ell V_{y_j}(a_0)$.
We enlarge the set $\{y_j\}$ to include the points
$y_i(\sigma)$ mentioned above; adjust the positions
of the cross-sections $\Sigma_{y_j}$ if necessary to ensure that
they are disjoint; and define the global cross-section
$\Xi=\bigcup_{j=1}^\ell \Sigma_{y_j}$ and its smaller
version $\Xi(a)=\bigcup_{j=1}^\ell \Sigma_{y_j}(a)$ for each
$a\in(0,1)$. We also set $\wh{\Xi(a)}=\bigcup_{j=1}^\ell
\wh{\Sigma_{y_j}(a)}$ for future reference.

  In what follows we modify the choices of $U_0$ and $T_0$.
  However, $V_{y_j}$, $\Sigma_{y_j}$ and $\Xi$ remain
  unchanged from now on and correspond to our current choice
  of $U_0$ and $T_0$.  All subsequent choices will be
  labeled $U_1\subset U_0$ and $T_1\ge T_0$.  In particular
  $U_1\subset V_0 \cup \bigcup_{j=1}^\ell V_{y_j}(a_0)$.
  We set $\delta_0=d(\partial\Xi,\partial\Xi(a_0))>0$ where
  $\partial\Xi(a)$ is the boundary of the submanifold
  $\Xi(a)$ of $M$, $a\in(0,1)$, and $\Xi=\Xi(1)$.

  For future use, for each $j=1,\dots,l$ we write
  $\Pi_j: \wh{V_{y_j}(a)} \to \wh{\Sigma_{y_j}(a)}$ for the projection
  along flow lines within the sub-flow box $\wh{V_{y_j}(a)}$, that is,
  $x\in\wh{V_{y_j}(a)}$ and $\Pi_jx\in \wh{\Sigma_{y_j}(a)}$ belong to
  the same flow line within $\wh{V_{y_j}(a)}$. We note that, since
  this a finite collection of smooth maps, there exists $L>0$ so that
  $\Pi_j$ is $L$-Lipschitz for all $j=1,\dots, l$.
  
  \subsubsection{The Poincar\'e map}
  \label{sec:poincare-map}

  By Theorem~\ref{thm:Ws}, for any $\delta>0$ we can choose
  $T_1\ge T_0$ such that $\diam \phi_t(W^s_x(\Sigma_{y_j}))<\delta$,
  for all $x\in \Sigma_{y_j}$, $j=1,\dots,\ell$ and $t>T_1$. We fix
  $T_1=T_1(\delta)$ for
  $L\cdot \delta=\delta_0=d(\partial\Xi,\partial\Xi(a_0))$ in what
  follows.

  We define
  $
  \Gamma_0=\{x\in\Xi:\phi_{T_1+1}(x)\in\bigcup_{\sigma\in\sing(G)\cap
    U_0}(\gamma^s_\sigma\setminus\{\sigma\})\}$ and
  $\Xi'=\Xi\setminus\Gamma_0$.  If $x\in\Xi'$, then $\phi_{T_1+1}(x)$
  cannot remain inside $V_0$ so there exists $t>T_1+1$ and
  $j=1,\dots,\ell$ such that $\phi_tx\in V_{y_j}(a_0)$.  Since
  $\epsilon_0<1$, there exists $t>T_1$ such that
  $\phi_tx\in\Sigma_{y_j}(a_0)$.

  For each $\Sigma=\Sigma_{y_j}\in\Xi$, we choose a
  center-unstable disk $W_\Sigma$ which crosses $\Sigma$ and is
  transversal to $\cW^s(\Sigma)$, that is, every stable leaf
  $W^s_x(\Sigma)$ 
  intersects $W_\Sigma$ transversely at only one point, for each
  $x\in\Sigma$.

  For every given fixed $x\in W_\Sigma\cap\Xi'$, we define
\begin{align*}
  R(x)=\phi_{\tau(x)}(x)
  \quad\text{where}\quad
  \tau(x)=\inf\left\{t>T_1:\phi_tx\in\bigcup_{j=1}^\ell
  \Sigma_{y_j}(a_0) \right\}.
\end{align*}
Thus, there exists $k\in\{1,\dots,\ell\}$ so that
\begin{align*}
  \tau_k(x)=\inf\left\{t>T_1:\phi_tx\in \Sigma_{y_k}(a_0) \right\}
\end{align*}
equals $\tau(x)$.  We note that by the choice of $T_1=T_1(\delta)$ we
have $\diam \phi_{\tau(x)}(W^s_x(\Sigma))<\delta=\delta_0/L$ and so
the disk $\phi_{\tau(x)}(W^s_x(\Sigma))$, although not necessarily
contained in any $\wh{\Sigma_{y_j}(a_0)}$, is certainly contained in
$\wh{V_{y_k}(a_0)}$ by construction and so
$\Pi_j\big( \phi_{\tau(x)}(W^s_x(\Sigma)) \big) \subset
\wh{\Sigma_{y_k}(a_0)}$.  Hence, we can define
\emph{for each $y\in\cW^s_x(\Sigma)$}
\begin{align}\label{eq:deftau}
  R(y)=\phi_{\tau(y)}(y)
  \quad\text{where}\quad
  \tau(y)=\inf\left\{t>T_1:\phi_ty\in
  \wh{\Sigma_{y_k}(a_0)}\right\};
\end{align}
see Figure~\ref{fig:subsection} for a sketch of this procedure. Note
that since $\wh{\Sigma_{y_k}(a_0)}\supset\Sigma_{y_k}(a_0)$, the above
definitions of $\tau(x)$ and $\tau(y)$ coincide for $y=x$ and, by the
Tubular Flow Theorem, the definition~\eqref{eq:deftau} provides a
\emph{smooth extension} of the previous definition for
$x\in W_\Sigma\cap\Xi'$ to the whole $W^s_x(\Sigma)$ and also to a
neighborbood of $W^s_x(\Sigma)$ in $\Xi'$.

\begin{remark}
  \label{rmk:crucial}
  Let $C$ be the non-empty connected component of the set
  $\big(R\mid_{W_\Sigma}\big)^{-1}(\Sigma_{y_k}(a_0))$ of $W_\Sigma$
  containing $x$. Then
  $R\mid_{W^s_C(\Sigma)}:W^s_C(\Sigma)\to\Sigma_{y_k}(a_0)$ is smooth
  in the open sub-cross-section
  $W^s_C(\Sigma):=\cup\{W^s_z(\Sigma):z\in C\}\subset \Sigma$.  In
  Figure~\ref{fig:subsection}, we sketch a situation where the
  connected component $C$ is strictly inside $W_\Sigma$.

  Moreover, the union of the connected components of
  $\big(R\mid_{W_\Sigma}\big)^{-1}(\Sigma_{y_k}(a_0)), k=1,\dots,\ell$
  covers $W_\Sigma$ except for the subset of points sent to the
  boundary of $\Xi'$.
\end{remark}

We define the topological foliation
$\cW^s(\wh{\Xi})=\bigcup_{j=1}^\ell \cW^s(\wh{\Sigma_{y_j}(a_0)})$ of
$\wh{\Xi(a_0)}$ with leaves $W^s_x(\wh{\Xi})$ passing through each
$x\in\wh{\Xi(a_0)}$.  From the uniform contraction of stable leaves
together with the choice of $T_1$, $\delta$ and $\delta_0$, we deduce
that
\begin{align*}
  \diam\left(\Pi\big(\phi_{\tau(x)}W^s_x(\wh{\Xi})\big)\right)<L\delta=\delta_0
\end{align*}
and then by the flow invariance of $\cW^s$ and the previous definition
of the Poincar\'e map $R$, we conclude that
\begin{align*}
  R(W^s_x(\wh{\Xi}))
  =
  \Pi\big(\phi_{\tau(x)}W^s_x(\wh{\Xi})\big)
  \subset W^s_{Rx}(\wh{\Xi}).
\end{align*}
This proves the following

\begin{proposition}
  \label{prop:invf}
  For big enough $T_1>T_0$,
  $R(W^s_x(\wh{\Xi}))\subset W^s_{Rx}(\wh{\Xi})$ for all
  $x\in\wh{\Xi}'=\wh{\Xi(a_0)}\setminus\Gamma_0$.
\end{proposition}


In this way we obtain a piecewise $C^1$ global Poincar\'e map
$R:\wh{\Xi}'\to\wh{\Xi(a_0)}$ with piecewise
$C^1$ roof function $\tau:\wh{\Xi}'\to[T_1,\infty)$, and deduce the
following standard result.

\begin{lemma}{\cite[Lemma 3.2]{ArMel18}} \label{lem:log} If
  the section $\wh{\Sigma_y(a_0)}$ contains no equilibria
  (i.e. $\Gamma_0\cap\wh{\Sigma_y(a_0)}=\emptyset$), then
  $\tau\mid_{\wh{\Sigma_y(a_0)}}\le T_1+2$. In general, there is $C>0$
  so that $ \tau(x)\le -C\log\dist(x,\Gamma_0)$ for all
  $x\in\wh{\Xi}'$; moreover, $\tau(x)\nearrow\infty$ as
  $\dist(x,\Gamma_0)\searrow0$.
\end{lemma}

We define
$\partial^s\wh{\Xi(a_0)}=\bigcup_{j=1}^\ell
\partial^s\wh{\Sigma_{y_j}(a_0)}$ and
$\Gamma_1=\{x\in \wh{\Xi}':R(x)\in\partial^s\wh{\Xi(a_0)}\}$ and then
set $\Gamma=\Gamma_0\cup\Gamma_1$. Clearly
$\Gamma_0\cap\Gamma_1=\emptyset$.

\begin{lemma}
  \label{le:Gamma1}
  \begin{enumerate}
  \item $\Gamma_0$ is a $d_s$-submanifold of $\Xi$ given by a
    finite union of stable leaves
    $W_{s_i}(\Xi), i=1,\dots,k$; and
  \item $\Gamma_1$ is a regular embedded $(d-2)$-topological
    submanifold foliated by stable leaves from $\cW^s(\Xi)$
    with finitely many connected components.
  \end{enumerate}
\end{lemma}

\begin{remark}
  \label{rmk:topsubmanifold}
  Note that $\Gamma_0$ is a (smooth) submanifold of $\Xi$
  with codimension $d_{cu}-1$, so it separates $\Xi$ only if
  $d_{cu}=2$; while $\Gamma_1$ is a regular topological
  codimension $1$ submanifold of $\Xi$ and so it separates
  $\Xi$.
\end{remark}

\begin{proof}
  It is clear that $W^s_x(\Xi)\subset \Gamma$ for all
  $x\in\Gamma$, so $\Gamma$ is foliated by stable leaves.
  We claim that $\Gamma$ is precisely the set of those
  points of $\Xi$ which are sent to the boundary of $\Xi$ or
  never visit $\Xi$ in the future.

  Indeed, if $x_0\in\wh{\Xi}'\setminus\Gamma_1$, then
  $R(x_0)=\phi_{\tau(x_0)}(x_0)\in \Sigma'$ for some
  $\Sigma'\in\wh{\Xi(a_0)}=\{\wh{\Sigma_{y_j}(a_0)}\}$.  For $x$ close
  to $x_0$, it follows from continuity of the flow that
  $R(x)\in\wt{\Sigma}'$ (with $\tau(x)$ close to $\tau(x_0)$).  Hence
  $x\in\wh{\Xi}'\setminus\Gamma_1$ and since
  $\wh{\Xi}'=\wh{\Xi}\setminus\Gamma_0$, then the claim is proved and,
  moreover, $\Gamma$ is closed.

  For item (1), we note that
  $\Gamma_0\subset\Xi\cap\phi^{-1}_{[0,T_1+1]}\big(\bigcup_\sigma
  \gamma^s_\sigma\big)$ and we may assume without loss of
  generality that the above union comprises only generalized
  Lorenz-like equilibria; cf.
  Remark~\ref{rmk:notLorenzlike}(2). Hence
  $T_w\gamma_\sigma^s=E^{cs}_w$ for
  $w\in\gamma_\sigma^s\setminus\{\sigma\}$; see
  Remark~\ref{rmk:notLorenzlike}(1). Thus $\Gamma_0$ is
  contained in the transversal intersection between a
  compact $(d_s+1)$-submanifold and a compact
  $(m-1)$-manifold, so $\Gamma_0$ is a compact
  differentiable $d_s$-submanifold of $M$ and $\Xi$. In
  addition, since $\Gamma_0$ is foliated by stable leaves
  which are $d_s$-dimensional, then $\Gamma_0$ has only
  finitely many connected components in $\Xi$.

  For item (2), note that for each $x\in\Gamma_1$ we have
  that $R(x)\in\partial\wh{\Sigma_j(a_0)}\subset\Sigma_j$. Thus
  there exists a neighborhood $W_x$ of $x$ in $\Xi$ and
  $V_{Rx}$ of $R(x)$ in $\Sigma_j$ so that
  $R\mid W_x: W_x\to V_{Rx}$ is a diffeomorphism. Hence
  $\Gamma_1\cap W_x=(R\mid W_x)^{-1}(V_x\cap
  \partial^s\wh{\Sigma(a_0)})$ is homeomorphic to a
  $(d_{cu}-2+d_s)$-dimensional disk. Moreover, this shows
  that the topology of $\Gamma_1$ is the same as the
  subspace topology induced by the topology of $\Xi$. We
  conclude that $\Gamma_1$ is a regular topological
  $(m-2)$-dimensional submanifold.
  
  It remains to rule out the possibility of existence of infinitely
  many connected components $\Gamma_1^k, k\in\ZZ^+$ of $\Gamma_1$ in
  $\Xi$. Since $\Xi$ contains finitely many sections only, then there
  exists cross-sections $\Sigma_j,\Sigma_i$ in $\Xi$ and, taking a
  subsequence if necessary, an accumulation set
  $\wt{\Gamma}=\lim_k \Gamma_1^k$ within $\close(\Sigma_j)$ so that
  $R(\Gamma_1^k)\subset\partial^s\wh{\Sigma_i(a_0)}$ for all $k\ge1$.
  By the continuity of the stable foliation, $\wt{\Gamma}$ is an union
  of stable leaves.

  We claim that the Poincar\'e times $\tau(x_k)$ for
  $x_k\in\Gamma_1^k, k\ge1$ are uniformly bounded from
  above. For otherwise the trajectory
  $\phi_{[0,\tau(x_k)]}(x_k)$ intersects $V_\sigma$ for some
  $\sigma\in\sing(G)\cap U$ and accumulates $\sigma$. Hence,
  by the local behavior of trajectories near saddles and the
  choice of the cross-sections near $V_\sigma$, we get that
  $\wt{\Gamma}\subset\wt{\Sigma_i(a_0)}$ is not contained in the
  boundary of the cross-section. This contradiction proves
  the claim. Let $T$ be an upper bound for $\tau(x_k)$.

  Then, for an accumulation point $x\in\wt{\Gamma}$ of $(x_k)_{k\ge1}$
  we have that the trajectories $\phi_{[0,T]}(x_k)$ converge in the
  $C^1$ topology (taking a subsequence if necessary) to a limit curve
  $\phi_{[0,T]}(x)$ and so
  $R(x)=\phi_{\tau(x)}(x)\in\partial^s\wh{\Sigma_i(a_0)}$. Thus we can
  find neighborhoods $W_x$ of $x$ and $V_{Rx}$ of $R(x)$ in $\Xi$ so
  that for arbitrarily large $m$ we have that
  $R\mid W_x: W_x\to V_{Rx}$ is a diffeomorphism and
  $\Gamma_1\cap W_x=(R\mid
  W_x)^{-1}(V_x\cap\partial^s\wh{\Sigma_i(a_0)})$, which contradicts
  the regularity of $\Gamma_1$ as topological submanifold.  This
  concludes the proof of item (2) and the lemma.
\end{proof}

From now on we set $\Xi''=\wh{\Xi(a_0)}\setminus\Gamma$.  Then
$\Xi''=S_1\cup\dots\cup S_k$ for some fixed $k\ge1$, where each $S_i$
is a connected \emph{smooth strip}, homeomorphic to either
\begin{itemize}
\item $\D^{d_{cu}}\times \D^{d_s}$, if
  $\Gamma_0\cap\close{S_i}=\emptyset$; or 
\item $\D^{d_{cu}}\times (\D^{d_s}\setminus\{0\})$, otherwise.
\end{itemize}
The latter are \emph{singular (smooth) strips}.

We note that $R\mid S_i:S_i\to\wh{\Xi(a_0)}$ is a diffeomorphism
onto its image, $\tau\mid S_i:S_i\to[T_1,\infty)$ is smooth
for each $i$, $\tau\mid S_i\le T_1+2$ on non-singular strips
$S_i$ and also on a neighborhood of
$\partial^s(S_i\cup\Gamma_0)$ for singular strips $S_i$.
The foliation $\cW^s(\Xi)$ restricts to a foliation
$\cW^s(S_i)$ on each $S_i$.

\begin{remark} \label{rmk:X} In what follows it may be necessary to
  increase $T_1$ leading to changes to $R$, $\tau$, $\Gamma$ and
  $\{S_i\}$ (and the constant $C$ in Lemma~\ref{lem:log}); see
  Remark~\ref{rmk:obsadj}. However, the global cross-sections
  $\Xi=\bigcup\Sigma_{y_j}$; $\Xi(a_0)$ and
  $\wh{\Xi(a_0)}=\bigcup\wh{\Sigma_{y_j}(a_0)}$ are fixed throughout
  the argument.
\end{remark}

\begin{remark}\label{rmk:existensions}
  Since $R$ sends $\Xi''$ into the subsections $\wh{\Xi(a_0)}$ of
  $\Xi=\Xi(1)$, there are \emph{smooth extensions}
  $\wt{R_i}: \wt{S_i}\to\Xi$ of $R\mid S_i:S_i\to\wh{\Xi(a_0)}$, where
  $\tilde S_i\supset \close(S_i)\setminus\Gamma_0$.
\end{remark}

\subsection{Hyperbolicity of the global Poincar\'e map}
\label{sec:hyperb-global-poinca}

We assume from now on that $\Lambda$ is a sectional hyperbolic
attracting set with $d_{cu}\ge2$ and proceed to show that, for large
enough $T_1>1$, the global Poincar\'e map $R:\Xi''\to\Xi$ is
\emph{piecewise uniformly hyperbolic} (with discontinuities and
singularities).

Let $S\in\{S_i\}$ be one of the smooth strips. Then there
are cross-sections $\Sigma$, $\wt\Sigma\in\Xi$ so that
$S\subset\Sigma$ and $R(\Sigma)\subset\wt\Sigma$.  The
splitting $T_{U}M=E^s\oplus E^{cu}$ induces the continuous
splitting $T\Sigma=E^s(\Sigma)\oplus E^u(\Sigma)$, where
$ E^s_x(\Sigma)=(E^s_x\oplus\RR\{G(x)\})\cap T_x{\Sigma} $
and $ E^u_x(\Sigma)=E^{cu}_x\cap T_x{\Sigma}$ for
$x\in\Sigma$; and analogous definitions apply to
$\wt\Sigma$.


\begin{proposition}\label{prop:secUH}
  The splitting $T\Sigma=E^s(\Sigma)\oplus E^u(\Sigma)$ is
  \begin{description}
  \item[invariant]
    $DR\cdot E^s_x(\Sigma) = E^s_{Rx}(\wt\Sigma)$ for all
    $x\in S$, and
    $DR\cdot E^u_x(\Sigma) = E^u_{Rx}(\wt\Sigma)$ for all
    $x\in\Lambda\cap S$.
  \item[uniformly hyperbolic] for each given
    $\lambda_1\in(0,1)$ there exists $T_1>0$ so that if
    $\inf \tau>T_1$, then
    $\|DR \mid E^s_x(\Sigma)\| \le \lambda_1$ and
    $\|\big(DR \mid E^u_x(\Sigma)\big)^{-1}\| \ge
    \lambda_1^{-1}$ for all $x\in S$ and $S\in\{S_i\}$.
  \end{description}
  Moreover, there exists $0<\wt{\lambda_1}<\lambda_1$ so
  that, for all $x$ on a non-singular strip $S$, or for $x$
  on a neighborhood of $\partial^s(S\cup\Gamma_0)$ of a
  singular strip $S$ we have
  $\wt{\lambda_1} < \|\big(DR \mid
  E^s_x(\Sigma)\big)^{-1}\|$ and
  $\|DR \mid E^u_x(\Sigma)\|<\wt{\lambda_1}^{-1}$.
\end{proposition}

\begin{proof}
  See \cite[Proposition 4.1]{ArMel18} with straightforward
  adaptation to use area expansion along each
  two-dimensional subspaces within $E^u_x(\Sigma)$ in order
  to obtain uniform expansion; cf. \cite[Lemma
  8.25]{AraPac2010}. The last statement follows from the
  boundedness of $\tau$ on the designated domains;
  cf. Lemma~\ref{lem:log}.
\end{proof}

For a given $a>0$, $x\in\Sigma$ and $\Sigma\in\Xi$ we define
the \emph{unstable cone field} at $x$ as
$ \cC^u_x(\Sigma,a)=\{w=w^s+w^u\in E^s_x(\Sigma)\oplus
E^u_x(\Sigma): \|w^s\| \le a \|w^u\| \}.  $


\begin{proposition} \label{prop:sec-cone} For any $a>0$,
  $\lambda_1\in(0,1)$, we can increase $T_1$ and shrink
  $U_1$ such that if $\inf\tau>T_1$ then
  $DR(x)\cdot \cC^u_x(S,a) \subset \cC^u_{Rx}(S',a)$; and
  $\| DR(x)w\| \ge \|\pi^uDR(x)w\|\ge \lambda_1^{-1} \|w\|$
  for all $w\in \cC^u_x(S,a)$ and all $x\in S$ and
  $S,S'\in\{S_i\}$ so that $fx\in S'$. Moreover
  $\| DR(x)w\| \le \wt{\lambda_1}^{-1}\|w\|$ for $x$ in a
  non-singular $S$ or $x$ in a neighborhood of
  $\partial^s(S\cup\Gamma_0)$ for a singular $S$.
\end{proposition}

\begin{proof}
  See \cite[Proposition 4.2]{ArMel18}, use $\wt{\lambda_1}$ from
  Proposition~\ref{prop:secUH} and an uniform bound for $\|\pi^u\|$
  depending on the angle between stable and center-unstable directions
  and the width of the cone.
\end{proof}

Considering the union of the smooth strips $S$, the previous
results shows that we obtain a global continuous uniformly
hyperbolic splitting $T\Xi''=E^s(\Xi)\oplus E^u(\Xi)$ in the
following sense.

\begin{theorem} \label{thm:global} For given $a>0$ and
  $\lambda_1\in(0,1)$ we obtain a global Poincar\'e map $f$
  so that the stable bundle $E^s(\Xi)$ and the restricted
  splitting
  $T_\Lambda\Xi''=E^s_\Lambda(\Xi)\oplus E^u_\Lambda(\Xi)$
  are $DR$-invariant; and
  $ DR\cdot \cC^u_x(\Xi,a)\subset \cC^u_{fx}(\Xi,a)$ and
  $\|\pi_u Df(x)w\|\ge \lambda_1^{-1}\|w\|$ for all
  $x\in\Xi''$ and $w\in\cC^u_x(\Xi,a)$.
\end{theorem}

\begin{remark}\label{rmk:extensions}
  The extensions $\wt{R_i}:\wt{S_i}\to\Xi$ mentioned in
  Remark~\ref{rmk:existensions} are such that on
  $\wt{S_i}\setminus\close S_i$ the map $\wt{R_i}$ behaves as $R$ in
  Propositions~\ref{prop:secUH} and~\ref{prop:sec-cone}. In
  particular,
  $\delta_1=d(S_i,\partial\wt{S_i}) \ge\wt{\lambda_1}\cdot
  d(\Xi(a_0),\Xi) = \wt{\lambda_1}\delta_0$.
\end{remark}

\subsection{Distance between points on distinct stable leaves in
  cross-sections}
\label{sec:distance-between-sta}

Our argument to prove expansiveness hinges on showing that the
distance between points on distinct stable leaves through points on
close by orbits must increase at a definite rate. For that we relate
distance between stable leaves on cross-sections with the distance
between their images on the quotient map.

In the codimension-two case, that is, if $d_{cu}=2$, we use the
one-dimensional central-unstable cone field restricted to the
cross-sections to obtain the following.

\begin{lemma}
\label{le:lengthversusdistance}
Let us assume that $d_{cu}=2$ and that $a>0$ has been fixed,
sufficiently small.  Then there exists a constant $\kappa$ such that,
for any pair of points $x, y \in \Sigma\in\Xi$, and any $cu$-curve
$\gamma:[0,1]\to\Sigma$ joining $x$ to some point of $W^s(y,\Sigma)$,
we have $\ell(\gamma) \le \kappa \cdot d(x,y)$, where
$\ell(\gamma)=\int_0^1\|\dot\gamma\|$ is the length of $\gamma$ in the
induced distance Riemannian distance on $\Sigma$.
\end{lemma}

\begin{proof}
  This is basically \cite[Lemma 6.18]{AraPac2010} from the
  singular-hyperbolic setting conveniently restated in the
  codimension-two setting.
\end{proof}

For this result it is crucial that the center-unstable cones
$C^{cu}_x(\Sigma,a)$ on cross-sections $\Sigma\in\Xi$ have
one-dimensional core, that is, they are cones around a certain
one-dimensional subspace of the tangent space $T_x\Sigma$; see the
left hand side of Figure~\ref{fig:1cucone}.

\begin{figure}[htpb]
  \centering
  \includegraphics[width=4cm]{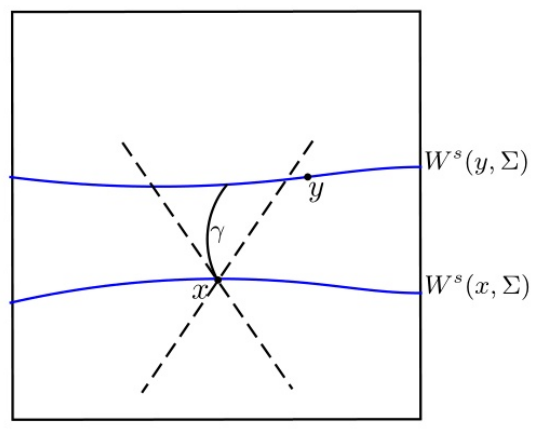}
  \includegraphics[width=4cm]{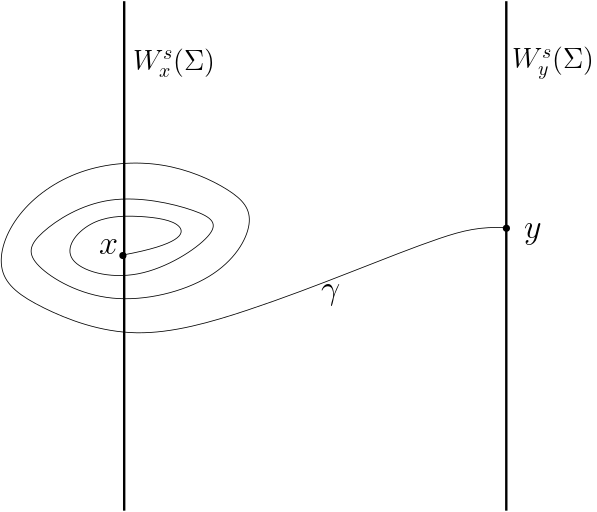}
  \caption{A sketch of a $cu$-curve connecting two stable leaves on a
    cross-section with $d_{cu}=2$, on the left; and a possible
    $cu$-curve connecting two stable leaves on a higher codimensional
    setting, on the right.}
  \label{fig:1cucone}
\end{figure}

\subsubsection{Construction of a $C^1$ local chart }
\label{sec:constr-c1-local}

For higher codimensions $d_{cu}>2$ there exist $cu$-curves
connecting two stable leaves which behave like tightly
curved helixes, loosing any relation between the lenght of a
general $cu$-curve and the distance between the leaves; see
the right hand side of Figure~\ref{fig:1cucone}. That is why
we assume the extra hypotheses of $1$-strong-dissipativeness
in the higher codimensional setting.

We have seen the topological foliation $\EuScript{W}^s$ of $U_0$
induces a topological foliation
$\EuScript{W}^s(\Sigma)=\{W^s_x(\Sigma)\}_{x\in\Sigma}$ on each
$\Sigma$.  For a $1$-strongly dissipative $C^1$ vector field $G$,
Theorem~\ref{thm:qdiss} guarantees that the foliation $\EuScript{W}^s$
is $C^1$, that is, the map
$$\gamma:U_0\to \textrm{Emb}^{1}(\D^{d_s},\Sigma),
\quad\text{such that}\quad \gamma(x)(0)=x \text{ and }
\gamma(x)(\D^{d_s})=W^s_x$$ in the notation of Theorem~\ref{thm:Ws}
satisfies that $(x,y)\in U_0\times\D^{d_s}\mapsto \gamma(x)(y)$
becomes a local $C^1$ diffeomorphism.

Considering the transversal disk $W_\Sigma$ in $\Sigma$, we have that
$W_\Sigma$ is diffeomorphic to $\D^{d_{cu}-1}$ and define the local
chart $\psi:\D^{d_{cu}-1}\times\D^{d_s}\to\Sigma$ by
$\psi(x,y)=\gamma(x)(y)$.  Thus
$\psi(\{x\}\times\D^{d_s})=\gamma_x(\D^{d_s})=W^s(x,\Sigma)$ and
$\psi$ is a $C^1$ chart of $\Sigma$.

\begin{lemma}
\label{le:legthvsdisthighdcu}
There are constants $K,L>0$ such that
$$L\cdot d_e(x,y)\leq d(W^s_x(\Sigma),W^s_y(\Sigma))\leq K\cdot d_e(x,y),$$
where $d$ and $d_e$ denote the Riemannian distance and the Euclidean
distance, respectively.
\end{lemma}

\begin{proof}
  By definition $d(W^s_x(\Sigma),W^s_y(\Sigma))=\inf_{\xi}\ell(\xi)$,
  where $\ell(\xi)$ is the length of the curve $\xi$ connecting some
  point of $W^s_x(\Sigma)$ to another point of $W^s_y(\Sigma)$ inside
  $\Sigma$. Since $\psi$ is a diffeomorphism, $\xi$ is the image of
  some curve $\tilde{\xi}\in \D^{d_{cu}}\times\D^{d_s}$ connecting
  $\{x\}\times\D^{d_s}$ to $\{y\}\times\D^{d_s}$ and
  $\ell(\xi)=\int_a^b\|D\psi(\tilde{\xi})\cdot\tilde{\xi}'(t)\|dt$.
  Then $\ell(\xi)\leq K\cdot\ell(\tilde{\xi})$ where
  $K=\sup_{x\in\Sigma}\|D\psi(x)\|$. Hence
\begin{align*}
  d\big(W^s(x,\Sigma),W^s(y,\Sigma)\big)
  &=
    \inf_{\xi}\ell(\xi)
    \leq K\inf_{\xi}\ell(\tilde{\xi})
  =
    K\cdot d_e(\{x\}\times\D^{d_s},\{y\}\times\D^{d_s})
    =
    K\cdot d_e(x,y).
\end{align*}
Analogously, $\ell(\xi)\geq L\cdot\ell(\tilde{\xi})$ with
$L=\inf_{x\in\Sigma} \|D\psi^{-1}(x)\|^{-1}$ and so 
\begin{align*}
  d\big(W^s(x,\Sigma),W^s(y,\Sigma)\big)
  &=
    \inf_{\xi}\ell(\xi)
    \geq
    L\inf_{\xi}\ell(\tilde{\xi})
    =
    L\cdot d_e( \{x\}\times\D^{d_s}, \{y\}\times\D^{d_s} )
    =
    L\cdot d_e(x,y)
\end{align*}
finishing the proof.
\end{proof}

\subsection{The Poincar\'e quotient map on a cross-section}
\label{sec:poincare-quotient-ma}

Recall the choice of a a center-unstable disk $W_\Sigma$ transversal
to $\cW^s(\Sigma)$ for each $\Sigma\in\wh{\Xi(a_0)}$, and consider the
projection $\pi:\Sigma\to W_\Sigma$ in each $\Sigma\in\wh{\Xi(a_0)}$ which
maps every $x\in\Sigma$ to the point $\pi(x)$ such that
$$W^s_x(\Sigma)\cap W_\Sigma=\{\pi(x)\}.$$

The quotiented cross-section is homeomorphic to $W_\Sigma$.  Since the
foliation $\cW^s(\Sigma)$ is preserved by $R$, that is,
$R(\cW^s_x(\Sigma))\subset \cW^s_{Rx}(\Sigma')$, where
$R(x)\in\Sigma'$, the Poincar\'e quotient map
$f:\cW_\Sigma\to\cW_{\Sigma'}$ is given by
$f\circ\pi(x)=\pi\circ R(x)$.

\subsubsection{The quotient map is expanding on the smooth strip}
\label{sec:quotient-map-expand}

In the $d_{cu}>2$ case, since the vector field is assumed to be
$1$-strongly dissipative, then we can use the $C^1$ local chart
$\psi$.

We have that
$\Sigma\cong\EuScript{D}^{d_{cu}-1}\times\EuScript{D}^{d_s}$,
the stable foliation is given by
$\big\{\{a\}\times\EuScript{D}^{d_s}\big\}$ and
$W_\Sigma\cong\EuScript{D}^{d_{cu}-1}$. Then, in these
coordinates, $\pi$ is given by the canonical projection on
$\RR^{d_{cu}-1}$,
$\pi:\EuScript{D}^{d_{cu}-1}\times\EuScript{D}^{d_s}\to\EuScript{D}^{d_{cu}-1}$
and $f:\EuScript{D}^{d_{cu}-1}\to\EuScript{D}^{d_{cu}-1}$
can be written $f=\pi\circ R|_{\EuScript{D}^{d_{cu}-1}}$. In
this setting the map $f$ is a piecewise $C^1$ map with
smooth domains on each projected strip $\pi(S_i)$.

\begin{lemma}
\label{le:quotientexp}
  Denote by $\pi(S_i)$ a smooth strip of $f$
  corresponding to the strip $S_i$ of $R$. There exists $\mu>1$ such
  that $\|Df(x)^{-1}\|\le\mu^{-1}<1$, for every $x\in \pi(S_i)$.
\end{lemma}

\begin{proof}
  Given a vector $v\in\RR^{d_{cu}-1}$, there exists a curve
  $\gamma:I\to\EuScript{D}^{d_{cu}-1}$ with $\gamma(0)=x$ and
  $\gamma'(0)=v$. This $\gamma$ is a $cu$-curve by construction and
$$Df(x)v=D(\pi\circ R)(x)(v)=D\pi(R(x))DR(x)v=\pi(R(x))DR(x)v.$$ 
Using the Proposition~\ref{prop:sec-cone}, we have:
$\|Df(x)v\|=\|\pi(R(x))DR(x)v\|\geq \lambda_1^{-1}\|v\|.$ Hence,
denoting $\mu=\lambda_1^{-1}$ we get $\|Df(x)^{-1}\|\le\mu^{-1}$.
\end{proof}

\begin{remark}
  \label{rmk:smoothextquotient}
  We also define extensions $\wt{f_i}:\pi(\wt{S_i})\to\Xi$ of
  $f\mid \pi(S_i)$ satisfying $\wt{f_i}\circ\pi=\pi\circ\wt{R_i}$
  which clearly have the same properties stated in
  Lemma~\ref{le:quotientexp}.
\end{remark}


\section{Proof of robust expansiveness}
\label{sec:proof-main-theorems}

Here we prove the main Theorems~\ref{mthm:principal1}
and~\ref{mthm:principal2}.

\emph{Arguing by contradiction, we assume that the flow is not
  expansive on $U_0$}, the trapping region containing $\Lambda$, that
is, there exists $\epsilon>0$ such that for all $\delta>0$, we can
find $x,y\in U_0$ and $h\in S(\RR)$ satisfying
$\dist(\phi_t(x),\phi_{h(t)}(y))\leq\delta$ and
$\phi_{h(t)}(y)\notin \phi_{[t-\epsilon,t+\epsilon]}(x)$
for all $t\in\RR$.
Then, we can take $\delta_n\searrow 0$, $x_n,y_n\in U_0$ and
$h_n\in S(\RR)$ such that for all $t\in\RR$
\begin{align}\label{eq:nexpansive}
  d(\phi_t(x_n),\phi_{h_n(t)}(y_n))\leq\delta_n
  \qand
  \phi_{h_n(t)}(y_n)\notin \phi_{[t-\epsilon,t+\epsilon]}(x_n).
\end{align}
Since the set of accumulation points is not empty, there exists some
regular point $z\in\Lambda$ which is accumulated by the sequence of
$\omega$-limit sets $\omega(x_n)$ in the following sense: there exists
$z_n\in \omega(x_n)$ for each $n\ge1$ such that $z_n\to z$.

Using that $\C$ is a cover of $\Lambda$ and $\omega(x)\subset\Lambda$,
we can assume without loss of generality that $z$ is inside some
$\Sigma^\delta\in\Xi^\delta$. This guarantees that
$d(z,\partial^s\Sigma)>\delta$.  We can now choose a neighborhood $V$
of $z$ contained in $\Sigma^\delta(\epsilon_0)$ for which there exists
$n_0>1$ such that $z_n\in V$ for all $n\geq n_0$.

Then, the orbit of $x_n$ returns infinitely often to a neighborhood of
$z_n$ which, on its turn, is close to $z$ and inside $V$. For this, we
can take $\delta_n$ small enough (if necessary) so that the orbit of
$y_n$ visits $V$ infinitely many times.

Let $t_n$ be the corresponding time to the first intersection between
the orbit of $x_n$ and $\Sigma^\delta$. Replacing $x_n$,$y_n$, $t$ and
$h_n$ by $x^{(n)}=\phi_{t_n}(x_n)$, $y^{(n)}=\phi_{h_n(t_n)}(y_n)$,
$\tilde{t}=t-t_n$ and
$\tilde{h}_n(\tilde{t})={h}_n(\tilde{t}+t_n)-h_n(t_n)$ we still have
\begin{align*}
  d(\phi_{\tilde{t}}(x^{(n)}),\phi_{\tilde{h}_n(\tilde{t})}(y^{(n)}))\leq \delta_n
\qand
  \phi_{{\tilde{h}}_n(\tilde{t})}(y^{(n)})\notin
  \phi_{[\tilde{t}-\epsilon,\tilde{t}+\epsilon]}(x^{(n)}).
\end{align*}
Moreover, by construction of $V$, we can prove the following.
\begin{proposition}\label{quatro}
  There exists $K>0$, depending only on the angle between $\Sigma$ and
  the direction of the flow (see figure \ref{vizi}), such that for
  every $n\geq 0$ there are sequences $(\tau_{n,j})$ (with
  $\tau_{n,0}=0$) and $(\upsilon_{n,j})$ such that
\begin{itemize}
\item $x_{n,j}=\phi_{\tau_{n,j}}(x^{(n)})\in\Sigma^\delta$ and $y_{n,j}=\phi_{\upsilon_{n,j}}(y)\in \Sigma^\delta$ for all
  $j\geq 0$, where $\upsilon_{n,j}=h(\tau_{n,j})+\epsilon_{n,j}$;
\item $\tau_{n,j}-\tau_{n,j-1}>T_1$;\,
  $|\upsilon_{n,j}-h(\tau_{n,j})|<K\delta_n$ and
  $d(x_{n,j},y_{n,j})<K\delta_n$.
\end{itemize}
\end{proposition}

\begin{proof}
  This is contained in \cite[Theorem 7.13]{APPV}. The proof does not
  use (co)dimension nor hyperbolicity assumptions.
\end{proof}

We will fix a convenient $n$ in what follows and write $x=x^{(n)}$ and
$y=y^{(n)}$. We observe that $\phi_t(x)$ and $\phi_s(y)$ are not in
the local stable manifold $W^s_{loc}(\sigma)$ of some
$\sigma\in Sing(\Lambda)$, for all $t,s\in\RR$. For otherwise, these
points could not return to $\Xi$ infinitely often.  Then
$R^j(x)$ and $R^j(y)$ are well-defined for all $j\geq 0$ and we write
$x_j=R^j(x)$ and $y_j=R^j(y)$ in what follows.

\begin{figure}[h]
  \begin{center}
 \includegraphics[width=6cm]{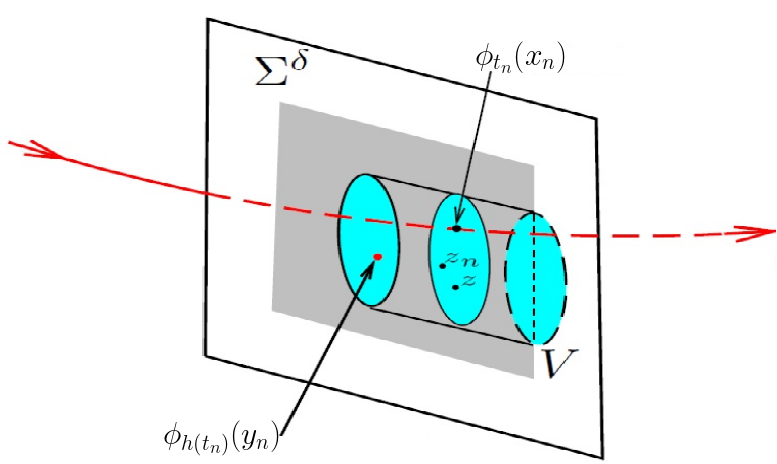}
\end{center}
\caption{The neighborhood $V$.}
\label{vizi}
\end{figure}

\begin{remark}\label{obsdist}
  Whenever $R(x_j)=\phi_{\tau(x_j)}(x_j)$ and
  $R(y_j)=\phi_{\tau(y_j)}(y_j)$ are in the same $\Sigma\in\Xi$, we
  can estimate as in Proposition~\ref{quatro} to ensure that
  $\tau(y_j)=\tau(x_j)+\epsilon_j$ and so
  $d(R(x_j),R(y_j))<K(\Sigma)\delta_n$.

  In general, we have that $x_{j+1}=R^j(x)\in\Xi(a)$ and we can find
  $\wt{\tau_{j+1}(y)}$ such that setting
  $\wt{y_{j+1}}=\wt{R}(y_j)=\phi_{\wt{\tau_{j+1}(y)}}(y)\in\Sigma\in\Xi$
  we get $x_{j+1}=R(x_j)=\wt{R}(x_j)\in\Sigma(a)$; and also
  $\wt{\tau(y)}=\tau(x_j)+\epsilon_j$.

  Since there exists finitely many sections in $\Xi$, we can take an
  uniform constant $K$ such that
  $d(\wt{R}(x_j),\wt{R}(y_j))<K\delta_n$ for all $\Sigma\in\Xi$ such
  that $R(x_j)$,$R(y_j)\in \Sigma$.

  Then $x_j=R^j(x)=\phi_{\tau(x_j)}(x)$ is well-defined for all
  $j\geq 0$ and we set $y_j=\wt{R}(y_{j-1})=\phi_{\tau(y_j)}(y)$ for
  all $j\ge1$ in what follows.
\end{remark}

The essential part of the proof of the main theorems is to deduce the
following.

\begin{theorem}
\label{thm:expansivepoincare}
Given $\epsilon_0>0$ there exists $\delta_0>0$ such that, if
$x,y\in\Xi$ 
satisfy
\begin{enumerate}
\item there exist $\tau_j$ such that $x_j=\phi_{\tau_j}(x)\in\Xi$ with
  $\tau_{j}-\tau_{j-1}>T_1$ for all $j\ge1$;
\item $d\big(\phi_tx,\phi_{h(t)}y\big) < \delta_0$ for all $t>0$ and
  some $h\in S(\RR)$;
\end{enumerate}
then there exists $j\ge1$ and
$\eta\in[\tau_j-\epsilon_0,\tau_j+\epsilon_0]$ such that
$\phi_{h(\tau_j)}y \in \cW^s_{\phi_\eta x}$.
\end{theorem}

We postpone the proof of this result to
Subsection~\ref{sec:proof-claim-refyxs} and deduce now the statements
of the main theorems.

\subsection{Proof of the main Theorems~\ref{mthm:principal1}
  and~\ref{mthm:principal2}}
\label{sec:conclus-proof}

We assume the conclusion of Theorem~\ref{thm:expansivepoincare} and
finish the proof of both main Theorems~\ref{mthm:principal1}
and~\ref{mthm:principal2}, proceeding as in \cite[Subsection
3.3.4]{APPV}.

We note the following geometric consequence of transversality of the
flow to the stable foliation in $U_0$.

\begin{lemma}\label{le:proximo}
  There exist small $\rho,c>0$, depending only on the flow, such that
  if $z_1, z_2, z_3$ are points in $U_0$ satisfying
  $z_3\in \phi_{[-\rho,\rho]}(z_2)$ and
  $z_2\in B(z_1,\rho)\cap\cW_{z_1}^{s}$, with $z_1$ away from any
  equilibria, then
  $ d(z_1,z_3) \ge c \cdot \max\{d(z_1,z_2),d(z_2,z_3)\}.  $
\end{lemma}

\begin{proof}
  This is a direct consequence of the fact that the angle between
  $E_x^{s}$ and the flow direction $G(x)$ is bounded from zero which,
  in its turn, follows from the fact that the latter is contained in
  the center-unstable sub-bundle $E^{cu}$; see e.g. \cite[Lemma
  3.2]{APPV} or \cite[Lemma 7.12]{AraPac2010} and the left hand side
  of Figure~\ref{fig:intersecaoWs}.
\end{proof}

\begin{remark}\label{rmk:cfieldsize}
  Since we may shrink the value of $c>0$ in Lemma~\ref{le:proximo}, we
  assume without loss of generality that $10c<\sup\{\|G(z)\|:z\in\Xi\}$.
\end{remark}

We fix $\epsilon_0=\epsilon$ as in \eqref{eq:nexpansive} and then consider
$\delta_0$ as given by Theorem~\ref{thm:expansivepoincare}.

\begin{remark}\label{rmk:obsadj}
  \begin{enumerate}
  \item We fix $T_1$ large enough so that the construction of the
    global Poincar\'e return map, described in
    Section~\ref{sec:preliminary-results}, provides
    $\lambda_1\in (0,1)$ satisfying
    $\mu^{-1}=\lambda_1<\min\{1/2,\kappa, L/K\}$, where the constants
    are given by Lemmas~\ref{le:legthvsdisthighdcu}
    and~\ref{le:quotientexp}. From now on we fix $\Xi$ and the smooth
    strips $\{S_i\}\subset\Xi$.
  \item We fix $n$ such that $\delta_n$ is sufficiently small
    according to the following conditions.
    \begin{itemize}
    \item $\delta_n<\delta_0$ and $\delta_n < c^2\cdot\rho<c\rho$.
  
    \item \emph{Suppose that $x_j$ and $y_j$ are in the same strip
        $S_i$ of $R$ and consequently $\hat{x_j}=\pi(x_j)$ and
        $\hat{y_j}=\pi(y_j)$ are in the same smooth domain of
        $\hat{S_i}=\pi(S_i)$ of $f$.} We can choose $\delta_n$ small
      enough so that, if the ball $B(x_j,K\delta_n)$ is not entirely
      contained in $S_i$, then $B(x_j,K\delta_n)$ shall be entirely
      contained in the smooth strip $\wt{S_i}$ of $\wt{R}$, the
      extension map of $R$; and $B(\hat{x_j},K\delta_n)$ is contained
      in $\pi(\wt{S_i})$ the extended smooth domain of $\wt{f}$.
  
    \item \emph{If $x_j$ and $y_j$ are not in the same smooth strip of
        $\Sigma$, then we can assume that $x_j$ and $y_j$ are in
        adjoining strips.}  Indeed, it is enough to take
      \begin{align*}
        2K\delta_n<\min\{d(S_j,S_k): S_j,S_k
        \text{  are  non-adjoining strips in  } \Xi\}.
      \end{align*}
      Consequently, $y_j$ belongs to the extended domain
      $\wt{S}$ which contains $x_j$.

    \end{itemize}
  \end{enumerate}

\end{remark}
Now we apply Theorem~\ref{thm:expansivepoincare} to $x=x_0=x^{(n)}$,
$y=y_0=y^{(n)}$, $x_j=R^j(x_0)=\phi_{\tau_j}(x_0)$ and $h=h_n$: where
hypothesis (1) corresponds to the choice of $\tau_{n,j}$ from
Proposition~\ref{quatro} and, with these choices, hypothesis (2)
follows by the choice of $x,y$ from~\eqref{eq:nexpansive}.

\begin{figure}[h]
\centering
\includegraphics[width=10cm]{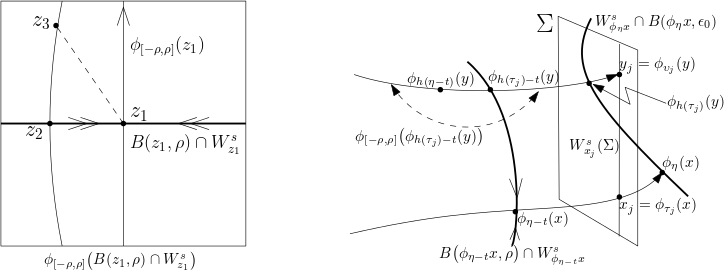}
\caption{\label{fig:intersecaoWs}Sketch of the setting of
  Lemma~\ref{le:proximo} on the left; and of the proof of
  Theorems~\ref{mthm:principal1} and~\ref{mthm:principal2} using
  Theorem~\ref{thm:expansivepoincare} on the right hand side.}
\end{figure}

Therefore, we obtain
$\phi_{h(\tau_j)}y\in B(\phi_\eta x,\epsilon_0)\cap\cW^{s}(\phi_\eta
x)$ for some $\eta\in[\tau_j-\epsilon_0, \tau_j+\epsilon_0]$ and
$j\ge1$. From the right hand side of \eqref{eq:nexpansive} we have
$\phi_{h(\tau_j)}y \neq \phi_\eta x$.  Hence, since the leaves of the
stable foliation are expanded under backward iteration, there exists a
maximum $\theta>0$ such that for all $0\le t\le \theta$ (see the right
hand side of Figure~\ref{fig:intersecaoWs})
$$
\phi_{h(\tau_j)-t}y\in B(\phi_{\eta-t}x,\rho)\cap
\cW^{s}_{\phi_{\eta-t}x}
\quad\text{and}\quad
\phi_{h(\eta-t)}y \in \phi_{[-\rho,\rho]}(\phi_{h(\tau_j)-t}y).
$$
Moreover, $x_j$ is close to $\Xi$ which is uniformly bounded away from
the equilibria, and then  $\|G(\phi_tx_j)\|\ge c$ for
$0\le t\le \theta$.  Since $\theta$ is maximum
\begin{align*}
  \text{either  }
  d\big(\phi_{h(\tau_j)-t}y,\phi_{\eta-t}x\big)&\ge\rho
  \text{  or  }
  d\big( \phi_{h(\eta-t)}y, \phi_{h(\tau_j)-t}y\big)\ge c\rho
  \text{  for  } t=\theta.
\end{align*}
Applying now Lemma~\ref{le:proximo} we deduce that
$ d\big(\phi_{\eta-t}x,\phi_{h(\eta-t)}y\big)\ge c^2\rho>\delta_n$
contradicting the choice of $x,y$ from~\eqref{eq:nexpansive}.  This
completes the proof of expansiveness for $\phi_t$ in the trapping
region $U_0$ of $\Lambda$ assuming
Theorem~\ref{thm:expansivepoincare}.

\subsubsection{On robustness of expansiveness}
\label{sec:robustn-expans}

To obtain robustness of expansiveness, we observe that
\begin{enumerate}
\item there exists a neighborbood $\V\subset\X^1(M)$ of $G$ in the
  family of all $C^1$ vector fields with the $C^1$ topology for which
  the family $\Xi$ of adapted cross-sections for $\Lambda$ and $G$
  remains a family of adapted cross-sections for
  $\Lambda_Y(U_0)=\cap_{t>0}\ov{\psi_t U_0}$ and all $Y\in\V$, where
  $\psi_t$ is the flow generated by $Y$.

  This is a consequence of $C^1$ closeness between $Y$ and $G$ and the
  continuity of the map $Y\in\V\mapsto\Lambda_Y(U_0)$ in the Hausdorff
  topology; this holds for every isolated set: see e.g. \cite[Lemma
  2.3]{AraPac2010}.
\item Consequently, the hyperbolicity constants for the global
  Poincar\'e return map $R_Y$ can be taken uniform on $Y\in\V$,
  including the threshold time $T_1$ and the value of $K$.
\item Moreover, the smooth strips of $R_Y$ are uniformly close in the
  Riemannian distance to the corresponding strips of $R=R_G$, and so
  $\epsilon_0$ and $\delta_0=\delta_n$ is the previous argument can
  also be taken uniformly on $Y\in\V$.
\end{enumerate}
Hence, Theorem~\ref{thm:expansivepoincare} holds for all $Y\in\V$ with
constant values of $\epsilon_0$ and $\delta_0$. This is enough to
conclude that expansiveness is robust for all sectional-hyperbolic
attracting sets in the setting of Theorems~\ref{mthm:principal1}
and~\ref{mthm:principal2}.

This completes the proof of Theorems~\ref{mthm:principal1}
and~\ref{mthm:principal2} assuming
Theorem~\ref{thm:expansivepoincare}.

\subsection{Proof of positive expansiveness}
\label{sec:proof-claim-refyxs}

Now we prove Theorem~\ref{thm:expansivepoincare}.  We first
assume the following.
\begin{claim}\label{yxsigma}
  For some $j\ge0$ we have $x_j\in\cW^s_{y_j}(\Xi)$.
\end{claim}
By the invariance and uniqueness of the stable foliation (given by
Theorem~\ref{thm:Ws}), this implies that $x_j\in \cW^s_{y_j}(\Xi)$ and
$y_j\in\cW^s_{x_j}(\Xi)$ for all $j\ge0$.

We postpone the proof of this claim and explain first, following
\cite[Section 7.2.7]{AraPac2010} and \cite[Section 3.3.4]{APPV}, how
Theorem~\ref{thm:expansivepoincare} follows from Claim~\ref{yxsigma}.

\begin{proof}[Proof of Theorem~\ref{thm:expansivepoincare}]
  Let $j\ge0$ be such that $y_j\in W^s(x_j,\Sigma)$. Then, according
  to Proposition~\ref{quatro} and Remark~\ref{obsdist}, we have
  $|\tau(y_j)-h(\tau(y_j))|=\epsilon_j<K\cdot\delta_0$ and, by
  construction of the stable foliation on cross-sections, there exists
  a small $\bar\epsilon>0$ such that $\phi_t(y_j)\in\cW^s_{x_j}$ for
  some $|t|<\bar\epsilon$.  Therefore the trajectory
  $\cO_y=\phi_{[\tau(y_j)-K\cdot\delta_0-\bar\epsilon, \tau(y_j) +
    K\cdot\delta_0 +\bar\epsilon]}(y)$ must contain
  $\phi_{h(\tau(y_j))}(y)$. We note that this holds for all
  sufficiently small values of $\delta_0>0$ fixed from the beginning.

  Let $\epsilon_0>0$ be given and let us consider the piece of the
  orbit $\cO_x:=\phi_{[\tau_j-\epsilon_0,\tau_j+\epsilon_0]}(x)$ and
  the piece of the orbit of $x$ whose stable manifolds intersect
  $\cO_y$, i.e.,
  \[ {\cO}_{xy}=\left\{ \phi_s(x) : \exists
      t\in[\tau(y_j)-K\cdot\delta_0-\bar\epsilon, \tau(y_j) +
      K\cdot\delta_0 +\bar\epsilon]\mbox{ s.t.  }
      \phi_t(y)\in\cW^{s}_{\phi_s(x)} \right\}.
  \]
  Since $\phi_t(y_j)\in \cW^s_{x_j}$ we conclude that $\cO_{xy}$ is a
  neighborhood of $x_j=\phi_{\tau_j}(x)$. Moreover, this neighborhood
  can be made as small as needed by letting $\delta_0$ and so
  $\bar\epsilon$ small enough. In particular this ensures that
  $\cO_{xy}\subset \cO_x$ and so
  $\phi_{h(\tau_j)}(y)\in \cup_{z\in\cO_x}\cW^s_z$.  As this finishes
  the proof of Theorem~\ref{thm:expansivepoincare} assuming
  Claim~\ref{yxsigma}.
\end{proof}

\subsection{Proof of the claim}
\label{sec:proof-claim}

We argue by contradiction, assumming that $y_j\notin W^s_{x_j}(\Xi)$
for all $j\ge1$ and split the argument into the codimension two case
and the higher codimension case. The goal is to show that the pairs
$x_j$ and $y_j$ are either in the same smooth strip of the global
Poincar\'e return map $R$, or else they are in the same extended smooth
strip of the extension of the global Poincar\'e map $\wt{R}$.

\subsubsection{The codimension-two case}
\label{sec:codimension-two-case}

Let us assume first that $x_j$ and $y_j$ are in the same strip $S_i$
of $R$ in some cross-section $\Sigma$ for some $j\ge1$.

We can consider a $cu$-curve $\gamma:[0,1]\to S$ such that
$\gamma(0)=x_j$ and $\gamma(1)\in\cW^s_{y_j}(\Sigma)$ and
\begin{itemize}
\item by Proposition~\ref{prop:invf}, we have invariance of the
  stable foliation inside cross-sections;
\item by Proposition~\ref{prop:secUH}, we have invariance and
  expansion of $cu$-cones under iteration of smooth domains.
\end{itemize}
Hence $\zeta=R\circ\gamma$ is another $cu$-curve contained in some
$\Sigma'\in\Xi$ such that $\zeta(0)=x_{j+1}\in\Sigma'$ and
$\zeta(1)\in\cW^s_{y_{j+1}}(\Sigma')$ and, moreover,
$\ell(\zeta)\ge\lambda_1^{-1}\ell(\gamma)$.  Since we can find a point
$\wh{y_{j+1}}\in\cW^s_{y_{j+1}}(\Xi)=\cW^s_{y_{j+1}}(\Sigma')$ so that
$d\big(x_{j+1},\cW^s_{y_{j+1}}(\Xi)\big) =
d\big(x_{j+1},\wh{y_{j+1}}\big)$, we use the estimate of
Lemma~\ref{le:lengthversusdistance} to arrive at
\begin{align*}
  d\big(x_{j+1},\wh{y_{j+1}}\big)
  \ge
  \kappa\cdot\ell(\zeta)
  \ge
  \frac{\kappa}{\lambda_1}\ell(\gamma)
  \ge
  \frac{\kappa}{\lambda_1}\cdot
  d\big(x_j,\cW^s_{y_j}(\Xi)\big)
\end{align*}
and so, see Figure~\ref{fig:expans-within-tube}
\begin{align}\label{eq:concl1}
d\big(x_{j+1},\cW^s_{y_{j+1}}(\Xi)\big) \ge 2\cdot
d\big(x_j,\cW^s_{y_j}(\Xi)\big).
\end{align}
Otherwise, if $x_j,y_j$ are not in the same smooth strip, then they
belong to adjoining smooth strips $S,S'$ of $R$, by the choices made
according to Remark~\ref{rmk:obsadj}, and $y_j$ belongs to
$B(x_j,K\delta_n)=B(x_j,K\delta_0)$ contained in the extended strip
$\wt{S}$ which is a smooth domain for $\wt{R}$. This prevents in
particular that the boundary between $S$ and $S'$ is a singular line,
since then $y_{j+1}$ and $x_{j+1}$ would be in distinct
cross-sections, which is impossible.
\begin{figure}[htpb]
  \centering
 \includegraphics[height=2cm]{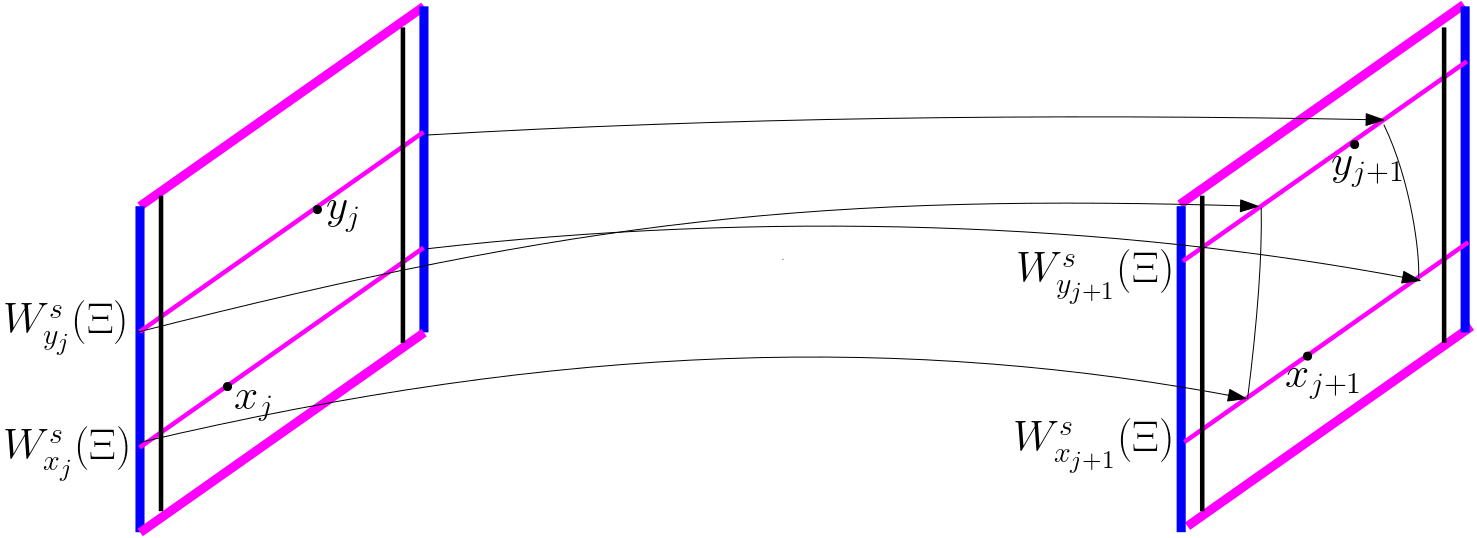}
 \caption{Expansion of distance between stable leaves.}
  \label{fig:expans-within-tube}
\end{figure}

We may now repeat the previous argument using the uniform expansion of
central-unstable curves to again conclude~\eqref{eq:concl1}.
Therefore, in both cases, we conclude by induction on $j\ge1$ that
\begin{align*}
  d\big(x_j,\cW^s_{y_j}(\Xi)\big)\ge 2^j \cdot
  d\big(x_0,\cW^s_{y_0}(\Xi)),
  \quad j\ge1.
\end{align*}
However, we have by assumption that $d\big(x_0,\cW^s_{y_0}(\Xi))>0$
and
\begin{align*}
d\big(x_j,\cW^s_{y_j}(\Xi)\big)\le d(x_j,y_j)\le\delta_0, \quad j\ge0.
\end{align*}
This yields a contradiction which proves that
$x_j\in \cW^s_{y_j}(\Xi)$ for some (and then, all) $j\ge0$.

\subsubsection{The higher codimensional case}
\label{sec:higher-codimens-case}
  
In the $d_{cu}>2$ case, let us assume again that $x_j$ and $y_j$ are
in the same strip $S_i$ of $R$ in some cross-section $\Sigma$.  Then
$\wh{x_j}=\pi(x_j)$ and $\hat{y_j}=\pi(y_j)$ are in the smooth domain
$\wh{S_i}=\pi(S_i)$ of $f$.  Let $S_j$ be the smooth domain where
$x_{j+1}$ lies and $\wh{S_j}=\pi(S_j)$ the corresponding domain of
$f$.

By the choices of constants according to Remark~\ref{rmk:obsadj}, we
get that $B=B(\wh{x_{j+1}},K\delta_n)$ is contained in the extended
strip $\wt{S_j}=\pi(\bar{S_j})$ and $B$ contains then line segment
$[\wh{x_{j+1}},\wh{y_{j+1}}]$.

Since $\wh{x_{j+1}}=f(\wh{x_j})$ and $\wh{y_{j+1}}=f(\wh{y_j})\in B$,
we can apply the Mean Value Theorem to $g=\big(f\mid_B\big)^{-1}$ and
use  Lemma~\ref{le:quotientexp} to get
\begin{align*}
  \|g\wh{x_{j+1}}-g\wh{y_{j+1}}\|_2
  \leq
  \sup_{z\in[\wh{x_{j+1}},\wh{y_{j+1}}]}\|Dg(z)\|
  \cdot \|\wh{x_{j+1}}-\wh{y_{j+1}}\|_2
  \le
  \mu^{-1}\|\wh{x_{j+1}}-\wh{y_{j+1}}\|_2,
\end{align*}
where $\|\cdot\|_2$ is the Euclidean norm. Hence
$\|f\wh{x_j}-f\wh{y_j}\|_2 \ge\mu\|\wh{x_j}-\wh{y_j}\|_2$.

In local coordinates (recall
Subsection~\ref{sec:poincare-quotient-ma}) $\cW^s_{x_j}(\Xi)$ and
$\cW^s_{y_j}(\Xi)$ correspond to $\{\wh{x_j}\}\times\D^{d_s}$ and
$\{\wh{y_j}\}\times\D^{d_s}$, respectively, for
$\wh{x_j},\wh{y_j}\in \D^{d_{cu}}$. We recall that
$f\circ\pi=\pi\circ R$ and using $d_e$ for the Euclidean distance, we
can write
$d_e\big(\cW^s_{x_j}(\Xi),\cW^s_{y_j}(\Sigma)\big) =
d_e(\wh{x_j},\wh{y_j}\big)$ in these local coordinates and also
\begin{align*}
  d\big(R\cW^s_{x_j}(\Xi),R\cW^s_{y_j}(\Xi)\big)
  &\geq
    d\big(\cW^s_{Rx_j}(\Xi), W^s_{Ry_j}(\Xi)\big)
    \ge
    L\cdot d_e(Rx_j,Ry_j)
  \\
  &=
   L \cdot d_e\big(f\wh{x_j},f\wh{y_j}\big)
  \geq
    L\mu\cdot d_e\big(\wh{x_j},\wh{y_j}\big)
  \\
  &\geq
    \frac{L}{K}\mu\cdot
    d\big(W^s_{x_j}(\Xi),W^s_{y_j}(\Xi)\big)
\end{align*}
where $\omega=L\mu/K=L/(K\lambda_1)>1$.

Finally, we analyze the setting where $x_j$ and $y_j$ are
not in the same smooth strip of $R$. As explained above, the
choice of $\delta_0=\delta_n$ ensures that $y_j$ belongs to
an adjoining strip to $x_j$.  By construction of the
cross-sections when $d_{cu}>2$, the intersection of local
stable manifolds of Lorenz-like singularities with $\Xi$ is
an isolated subset $\Xi\cap\Gamma_0$ in the interior of some
cross-sections. Hence
$y_j\in B=B(x_j,K\delta_0)\subset\wt{S}$ where $\wt{S}$ is
the extension of the smooth domain of $R$ containing $x_j$.

Therefore, $y_{j+1}\in B(x_{j+1},K\delta_0)\subset\wt{S}$ for some
extension of the smooth domain of $\Xi(a)$ containing $x_{j+1}$ and
$h=\big(\wt{R}\mid_B\big)^{-1}$ is well-defined. Moreover, we may now
consider the inverse of the corresponding quotient map
$g=\big(\wt{f}\mid_{\pi(B)}\big)^{-1}$ and apply the same argument as
before. We conclude by induction on $j\ge1$ that
\begin{align*}
  d(x_j,y_j)
  \ge
  d\big(R\cW^s_{x_j}(\Xi),R\cW^s_{y_j}(\Xi)\big)
  \ge
  \omega^j \cdot d\big(W^s_{x_0}(\Xi),W^s_{y_0}(\Xi)\big), \quad j\ge1
\end{align*}
and by assumption we again have
$d\big(W^s_{x_0}(\Xi),W^s_{y_0}(\Xi)\big)>0$ and
$d(x_j,y_j)\le\delta_0$ for all $j\ge0$.

This yields a contradiction and completes the proof of
Claim~\ref{yxsigma}.

\section{Sectional-hyperbolicity for a homogeneous attracting set}
\label{sec:singhypattracting}

Here we show that each homogeneous vector field in a trapping region
is necessarily sectional-hyperbolic.

\begin{theorem}
\label{thm:attracting-Lorenz-hyp-sing}
Let us assume that, for a $C^1$ neighborhood $\U$ of the
vector field $G$ in the space $\X^1(M^m)$ of $C^1$ vector
fields of a $d$-dimensional manifold, there exists an integer
$i\ge1$ so that $1<i+1<m$ and
\begin{description}
\item[(H1)] all periodic orbits in the trapping region $U$ are
    hyperbolic of saddle-type with index $i$; and
\item[(H2)]  the equilibria in $U$ are all generalized
  Lorenz-like with index $i$ or $i+1$.
\end{description}
Then the attracting set $\Lambda_G(U)=\bigcap_{t>0}\ov{\phi_t(U)}$ is
sectional-hyperbolic (where $\phi_t$ is the flow generated
by $G$).
\end{theorem}

The strategy is to assume robust hyperbolicity of periodic
orbits in the trapping region and use the techniques in the
proof of the main result from Morales, Pacifico and Pujals
\cite{MPP04} extended to higher-dimensional manifolds in
\cite{MeMor08} (see also \cite[Chapter 5]{AraPac2010}) to
deduce that the non-wandering subset
$\Omega_{\Lambda}=\Lambda\cap\Omega(G)$ of the attracting
set $\Lambda$ is sectional-hyperbolic.

We first show, in Subsection \ref{sec:parhypnonwand}, that
from sectional-hyperbolicity for $\Omega_{\Lambda}$ we
deduce that $\Lambda$ is sectional-hyperbolic. Then, in
Subsection \ref{sec:singhypnonwand}, we explain how robust
hyperbolicity of periodic orbits suffices to obtain
sectional-hyperbolicity for $\Omega_{\Lambda}$.

\subsection{Singular-hyperbolicity from the non-wandering
  set}
\label{sec:parhypnonwand}

Here we show that, if $\Lambda$ is the maximal forward
invariant set of a trapping region $U$, then it is enough to
prove that $\Lambda\cap\Omega(G)$ is sectional-hyperbolic to
conclude that the attracting set $\Lambda$ is
sectional-hyperbolic: this is due to compactness of $\Lambda$
and the uniform bounds of partial hyperbolicity.

\begin{proposition}
\label{pr:nonwander}
Let $\Lambda$ be the maximal forward invariant set of a trapping
region $U$, that is, $\Lambda=\cap_{t>0}\ov{\phi_t(U)}$ for a
$C^1$ vector field $G$. If
$\Omega_{\Lambda}:=\Omega(G)\cap\Lambda$ is sectional-hyperbolic,
then $\Lambda$ is sectional-hyperbolic.
\end{proposition}

\begin{proof}
  This follows almost immediately from the main theorem from Arbieto
  \cite{arbieto2010}.  Indeed, the subset $\Omega_{\Lambda}$ has total
  probability, since the non-wandering set contains the set of
  recurrent points and this set has full measure with respect to any
  invariant probability measure, by the Poincar\'e Recurrence
  Theorem. Hence, the assumptions of the Proposition ensure that, on
  the forward invariant open set $U$, there exists a subset of total
  probability which is sectional-hyperbolic (since $\Omega_{\Lambda}$
  is assumed to be sectional-hyperbolic). Thus, according to
  \cite{arbieto2010}, the maximal invariant subset of $U$ is
  sectional-hyperbolic. This maximal invariant subset is precisely the
  attracting set $\Lambda$.
\end{proof}

\subsection{Sectional-hyperbolicity of the non-wandering set
  from robust periodic hyperbolicity}
\label{sec:singhypnonwand}

Here we explain how we can obtain sectional-hyperbolicity for
the subset $\Omega_{\Lambda}$ from the assumption that
periodic orbits are $C^1$ robustly hyperbolic. The following
theorem together with Proposition~\ref{pr:nonwander} directly
imply Theorem~\ref{thm:attracting-Lorenz-hyp-sing}.

\begin{theorem}
\label{thm:attracting-hyp-sing}
Let us assume that for a $C^1$ neighborhood $\U$ of $G$ in
the space of $C^1$ vector fields the assumptions $(H1)$ and $(H2)$
in the statement of
Theorem~\ref{thm:attracting-Lorenz-hyp-sing} are valid.
Then the non-wandering part of the attracting set
$\Omega_{\Lambda}(G)$ is sectional-hyperbolic.
\end{theorem}

\begin{proof}
  We start by noting that, under the assumptions of
  Theorem~\ref{thm:attracting-hyp-sing}, for each periodic
  point $p\in\per(Z)$ for $Z\in\U$, we have that the tangent
  bundle of $M$ over $p$ can be written as
  $T_pM = E_p^s \oplus (\RR\cdot Z(p)) \oplus E_p^u$. Here
  ${E}_p^s$ is the eigenspace associated to the contracting
  eigenvalue of $D\psi_{t_p}(p)$; ${E}_p^u$ is the
  eigenspace associated to the expanding eigenvalue of
  $D\psi_{t_p}(p)$; and we write $t_p$ for the (minimal)
  period of $p$. Moreover, the dimensions of these
  subbundles are fixed, $d_s=\dim E^s_p=i$ and
  $d_{cu}-1=\dim E^u_p=m-i-1$, independently of $p$ and $Z$.

  We note that the possible presence of equilibria in
  $\ov{\per(Z)\cap\Lambda_Z(U)}$ is an obstruction for the
  extension of the stable and unstable bundles ${E}^s$ and
  ${E}^u$ to $\ov{\per(Z)\cap\Lambda_Z(U)}$.  Indeed, near a
  singularity, the angle between either ${E}^u$ and $Z$, or
  ${E}^s$ and $Z$, might be arbitrarily small.  To bypass
  this difficulty, we introduce the following notion.

    Given $Z\in \U$ define for any $p \in
    \per(Z)\cap\Lambda_Z(U)$ the splitting
    \begin{equation*}
      \label{s11}
      T_pM = E^{s,Z}_p \oplus  E^{cu,Z}_p,
      \quad
      \text{where}\quad
      E^{cu,Z}_p:=(\RR\cdot Z(p)) \oplus {E}_p^u.
    \end{equation*}
    Moreover we define a splitting over $\per(Z)
    \cap\Lambda_Z(U)$ by
    \begin{align*}
      T_{\per(Z)\cap\Lambda_Z(U)}M=\bigcup_{p\in
        \per(Z)\cap\Lambda_Z(U)} (E^{s,Z}_p \oplus
      E^{cu,Z}_p).
    \end{align*}
    In addition, we define the subspace
    $E^{cs,Z}_p:=E^{s,Z}_p\oplus (\RR\cdot Z(p))$ of the
    tangent space at $p\in\per(Z)\cap\Lambda_Z(U)$ which
    gives another bundle over $\per(Z)\cap\Lambda_Z(U)$.
  We denote the restriction of $D\psi_t(p)$ to $E^{s,Z}_p$
  (respectively $E^{cu,Z}_p$) by $D\phi_t\mid E^s_p$
  (respectively $D\psi_t\mid E^{cu}_p$) for $t\in\RR$ and
  $p\in \per(Z)\cap\Lambda_Z(U)$.

  We now prove that the splitting $E^{s,Z}_p \oplus E^{cu,Z}_p$ over
  $\per(Z)\cap\Lambda_Z(U)$ defined above is a
  $D\psi_t$-invariant and uniformly dominated splitting along periodic
  points with large period.

  \begin{theorem}{\cite[Theorem 5.37, Section
      5.4.1]{AraPac2010}}
    \label{dp}
    Given $G\in{\U}$, there are a neighborhood ${\V}\subset
    {\U}$ and constants $0<\lambda <1$, $c >0$, and $T_0 >0$
    such that, for every $Z \in {\V}$, if $p \in \per(Z)
    \cap\Lambda_Z(U)$, $t_p>T_0$ and $T>0$, then $ \|D\psi_T
    \mid {E}^s_p\| \cdot \|D\psi_{-T} \mid {E}^{cu}_{\psi_T(p)}\|
    < c\cdot \lambda^T.  $
  \end{theorem}

  The proof of this theorem is based on the following pair
  of results.

  Theorem~\ref{thm:pp} establishes, first, that the periodic
  points are uniformly hyperbolic, i.e., the periodic points
  are of saddle-type and the Lyapunov exponents are
  uniformly bounded away from zero.  Secondly, the angle
  between the stable and the unstable eigenspaces at
  periodic points are uniformly bounded away from zero.

  \begin{theorem}{\cite[Theorem 5.38, Section
      5.4]{AraPac2010}}
    \label{thm:pp}
    Given $G\in \U$, there are a neighborhood $\V\subset U$
    of $G$ and constants $0<\lambda < 1$ and $c > 0$, such
    that for every $Z\in\V$, if $p \in
    \per(Z)\cap\Lambda_Z(U)$ and $t_p$ is the period of $p$
    then
    \begin{enumerate}
    \item
      \begin{enumerate}
      \item $\quad\|D{\psi_{t_p}} \mid
        {E}^s_p\|<{\lambda}^{t_p}$ (uniform contraction on
        the period)
      \item $\quad\|D\psi_{-t_p}\mid {E}^u_p\|<{\lambda}^{t_p}$
        (uniform expansion on the period)\,.
      \end{enumerate}
    \item $\angle{E^{cs}_p,E^{cu}_p} > c$ (angle uniformly
      bounded away from zero between center-stable and
      center-unstable directions).
    \end{enumerate}
  \end{theorem}

  Theorem~\ref{thm:p1} is a strong version of item 2 of
  Theorem~\ref{thm:pp}.  It establishes that, at periodic
  points, the angle between the stable and the central
  unstable bundles is uniformly bounded away from zero.

  \begin{theorem}{\cite[Theorem 5.38, Section
      5.4]{AraPac2010}}
    \label{thm:p1}
    Given $G\in \U$ there are a neighborhood $\V\subset\U$
    of $G$ and a positive constant $C$ such that for every
    $Z\in\V$ and $p \in \per(Z)\cap\Lambda_Z(U)$ we have
    angles uniformly bounded away from zero:
    $\angle{E^{s}_p, E^{cu}_p} > C$.
  \end{theorem}

  We can assume without loss of generality that all the
  stated properties in previous results hold uniformly for
  all elements of $\per(Z)$ and $Z\in\U$ since, for each
  fixed $T>0$, hyperbolic periodic orbits with period at
  most $T$ are isolated and thus finitely many by relative
  compactness of $U$.

  The arguments of the proofs are as follows; see also
  \cite[Section 3]{MeMor08} and \cite[Section 5.4 and Remark
  5.35]{AraPac2010}.
  \begin{itemize}
  \item If Theorem \ref{dp} fails, then we can create a
    periodic point for a nearby flow with the angle between
    the stable and the central unstable bundles arbitrarily
    small. This yields a contradiction to Theorem
    \ref{thm:p1}.  In proving the existence of such a
    periodic point for a nearby flow we use Theorem
    \ref{thm:pp}. The arguments are presented in detail in
    \cite[Section 5.4.3]{AraPac2010}.
  \item Assuming Theorem \ref{dp}, we establish the extension of the
    splitting $E^{s,Z}_p \oplus E^{cu,Z}_p$ over
    $\per(Z)\cap\Lambda_Z(U)$ to a uniformly dominated splitting
    defined over all of $\Omega_{\Lambda}(G)$. This will be explained
    in the following subsection \ref{sec:domsplitnonwand}
  \item Afterwards, with the help of Theorem \ref{thm:pp},
    we can show that $E^s$ is uniformly contracting and
    that $E^{cu}$ is volume expanding. \emph{Hence}
    $\Omega_{\Lambda}(G)$ \emph{is a singular-hyperbolic
      set, as claimed} in the statement of Theorem
    \ref{thm:attracting-hyp-sing}. This can be done
    precisely as detailed in \cite[Section 5.3]{AraPac2010}:
    we show that the opposite assumption leads to the
    creation of periodic points for flows near to the
    original one with arbitrarily small contraction
    (respectively expansion) along the stable (respectively
    unstable) bundle, contradicting the first part of
    Theorem \ref{thm:pp}.
  \end{itemize}

  Finally, the proof of Theorem \ref{thm:pp} is presented in
  \cite[Sections 5.4.4 through 5.5.5]{AraPac2010} using the assumption
  all periodic orbits in $U$ are hyperbolic of saddle-type and all
  equilibria in $U$ are Lorenz-like, for all $Z\in\U$.  All of these
  facts together complete the proof of Theorem
  \ref{thm:attracting-hyp-sing}.
\end{proof}

  \subsubsection{Dominated splitting over the non-wandering
    part of the attracting set}
  \label{sec:domsplitnonwand}

  Here we induce a dominated splitting over $\Omega_{\Lambda}(Z)$
  using the dominated splitting $E^{s,Z}_p \oplus E^{cu,Z}_p$ over
  $\per(Z)\cap\Lambda_Z(U)$ over for flows near $G$, defined before on
  periodic orbits.

  On the one hand, since $\Lambda_Z(U)$ is an attracting
  set for every vector field $Z$ which is sufficiently
  $C^1$ close to $G$ we can assume, without loss of
  generality, that for all $Z \in \V$ and $p \in
  \per(Z)$ with ${\cO}_Z(p)\cap U \neq \emptyset$, we
  have ${\cO}_Z(p)\subset
  \Lambda_Z(U)\cap\Omega(Z)=\Omega_{\Lambda}(Z)$.

  On the other hand, every point of
  $\hat\Omega_{\Lambda}(G):=
  \Omega_{\Lambda}(G)\setminus(\per(G)\cup\sing(G))$
  is approximated by a periodic orbit of a $C^1$ nearby
  flow, by the Closing Lemma; see e.g. Pugh \cite{Pugh67} or
  Arnaud \cite{MCA98} for a more recent exposition.

  In addition, the remaining set
  $\Omega_{\Lambda}(G)\setminus\hat\Omega_{\Lambda}(G)$ is formed
  either by periodic points in $U$, which we assume are hyperbolic of
  saddle-type with index $i$ , or by equilibria, which we assume are
  Lorenz-like with index $i$ or $i+1$.  Hence all points of
  $\Omega_{\Lambda}(G)$ are either critical elements of $G$ or
  approximated by periodic orbits.

  More precisely, given $Z \in \V$, let $K(Z)
  \subset\hat\Omega_{\Lambda}(Z)$ be such that $\psi_t(x)
  \notin K(Z)$ for all $x\in K(Z)$ if $t \neq 0$.  In
  other words, $K(Z)$ is a set of representatives of the
  quotient $\hat\Omega_{\Lambda}(Z)/ \sim$ by the
  equivalence relation $x \sim y \iff x \in \cO_Z(y)$.
  From this, to induce an invariant splitting over
  $\Omega_{\Lambda}(G)$ it is enough to do it over
  $\hat\Omega_{\Lambda}(G)$. For this we proceed as
  follows.

  Since $\hat\Omega_{\Lambda}(G)\subset \Omega(G)$, then we
  can use the Closing Lemma: for any $x \in K(G)$ there
  exist
  \begin{itemize}
  \item a sequence $Z_n$ of vector fields in $M$ such
    that $Z_n \to G$ in the $C^1$ topology of
    vector fields; and
  \item $z_n \to x$ such that $z_n \in
    \per(Z_n)$.
  \end{itemize}
  We can assume without loss of generality that $Z_n \in
  \V$ for all $n$.  In particular $\cO_{Z_n}(y_n)\subset
  \Lambda_{Z_n}(U)\cap\Omega(Z)=\Omega_{\Lambda}(Z)$.
  Moreover, since $x\in K(G)$ is not periodic, we can also
  assume that the periods of $z_n$ are $t_{z_n} > T_0$
  for all $n$. Hence these periodic orbits admit a uniform
  dominated splitting whose features can be passed to the
  orbits of $\hat\Omega_{\Lambda}$ in the limit.

  More precisely, let us take a converging subsequence
  $E^{s,Z^{n_k}}_{z_{n_k}} \oplus E^{cu,Z^{n_k}}_{z_{n_k}}$
  and define
  $ E^{s,G}_x = \lim_{k
    \rightarrow\infty}E^{s,Z^{n_k}}_{z_{n_k}}$ and
  $E^{cu,G}_x= \lim_{k \rightarrow
    \infty}E^{cu,Z^{n_k}}_{z_{n_k}}.$ Since
  $E^{s,{Z_n}} \oplus E^{cu,Z_n}$ is a $(c,\lambda)$
  dominated splitting for all $n$, then this property is
  also true for the limit $E^{s,G}_x \oplus E^{cu,G}_x$.
  Moreover $\dim( E^{s,G}_x)=d_s$ and
  $\dim(E^{cu,G}_x)=d_{cu}$ for all $x\in\Lambda_G(U)$.

  Finally define $E^{s,G}_{\phi_t(x)} := D\phi_t(E^{s,G}_x)$ and
  $E^{cu,G}_{\phi_t(x)} := D\phi_t(E^{cu,G}_x)$ along $\phi_t(x)$ for
  $t\in\RR$. Since for every $n$ the splitting over
  $\{p\in\per(Z_n)\cap\Lambda_{Z_n}(U): t_p\ge T_0\}$ is $(c,\lambda)$
  dominated, it follows that the splitting defined above along $G$
  orbits of points in $K(G)$ is also $(c,\lambda)$-dominated.
  Moreover, we also have that $\dim E^{s,G}_{\phi_t(x)}=d_s$ and
  $\dim E^{cu,G}_{\phi_t(x)}=d_{cu}$ for all $t \in \RR$.

  This provides the desired extension of a dominated
  splitting to $\hat\Omega_{\Lambda}(G)$ and also to
  $\Omega_{\Lambda}(G)$, since the critical elements of
  $G$ in $U$ are
  \begin{itemize}
  \item either a periodic orbit with index $i$ or a
    generalized Lorenz-like singularity with the index $i+1$, in which
    case it already has a compatible dominated splitting;
  \item or a generalized Lorenz-like singularity with the same index
    $i$ of the periodic orbits, in which case this singularity is not
    in the $\omega$-limit set of any point of the attracting set.
  \end{itemize}

  The latter case above is treated in
  Proposition~\ref{prop:generaLorenzlike}.

  We denote by $E^s\oplus E^{cu}$ the splitting over
  $\Omega_{\Lambda}(G)$ obtained in this way. Note that if
  $\sigma \in\sing(G)\cap\Omega_{\Lambda}(G)$ has index $d_s+1$, then
  $E^s_\sigma$ is the direct sum of the eigenspaces $E^{ss}_\sigma$
  associated to the strongest contracting eigenvalues of
  $DG_{\sigma}$, and $E^{cu}_\sigma$ is $d_{cu}$-dimensional
  eigenspace associated to the remaining eigenvalues of $DG_{\sigma}$.
  This follows from the uniqueness of dominated splittings;
  see~\cite{Do87,Man88}.

  Since this splitting is uniformly dominated, we deduce
  that $E^s\oplus E^{cu}$ depends continuously on the points
  of $\Omega_{\Lambda}(G)$ and also on the vector field $G$
  in $\U$.

\subsection{Robust expansive attractors and sectional-hyperbolicity}
\label{sec:robust-expans-attrac}

Here we prove Corollary~\ref{mcor:starattractsechyp}. For that we need
to recall some results from~\cite{SGW14} on chain recurrent classes of
star vector fields.

Let $\phi_t$ be the flow generated by the vector field $G$. For any
$\varepsilon>0, T>0$, a finite sequence $\{x_i\}_{i=0}^n$ of points in
the ambient space is an \emph{$(\varepsilon,T)$-chain} of $G$ if there
are $t_i\ge T$ such that $d(\phi_{t_i}(x_i),x_{i+1})<\varepsilon$ for
all $0\le i\le n-1$.

A point $y$ is \emph{chain attainable} from $x$ if there exists $T>0$
such that for any $\varepsilon>0$, there is an $(\varepsilon,T)$-chain
$\{x_i\}_{i=0}^n$ with $x_0=x$ and $x_n=y$.  If $x$ is chain
attainable from itself, then $x$ is a \emph{chain recurrent
  point}. The set of chain recurrent points is the \emph{chain
  recurrent set} of $G$, denoted by ${\rm CR}(G)$.  Chain
attainability is a closed equivalence relation on ${\rm CR}(G)$.

For each $x\in{\rm CR}(G)$, the
equivalence class $C(x)$ (which
is compact) containing $x$ is
the \emph{chain recurrent class
  of $x$.}  A chain recurrent
class is {\it trivial} if it
consists of a single critical
element. Otherwise it is {\it
  nontrivial}.

Since every hyperbolic critical element $c$ of $G$ has a well-defined
continuation $c_Y$ for $Y$ close to $G$, the chain recurrent class $C(c)$ also
has a well-defined continuation $C(c_Y, Y)$.

A compact invariant set $\Lambda$ is called {\it chain transitive} if for every
pair of points $x,y\in \Lambda$, $y$ is chain attainable from $x$, where all
chains are chosen in $\Lambda$. Thus a chain recurrent class is just a maximal
chain transitive set, and every chain transitive set is contained in a unique
chain recurrent class.

Given $\sigma\in\sing(G)$ such that $C(\sigma)$ is non-trivial and $G$
is a star vector field, then we define the saddle-value of $\sigma$ as
\begin{align*}
  \sv(\sigma)=\lambda_s+\lambda_{s+1}
\end{align*}
where the Lyapunov exponents of $\phi_t\sigma$ are
$
\lambda_1\leq\cdots\leq\lambda_s<0<\lambda_{s+1}\leq\cdots\leq\lambda_m.
$ According to \cite[Lemma 4.2]{SGW14} if $C(\sigma)$ is nontrivial
for a star vector field, then $\sv(\sigma)\neq0$. We can now define
the \emph{periodic index} $\indx_p(\sigma)$ of $\sigma$ as
\begin{align*}
  \indx_p(\sigma)
  =
  \begin{cases}
    s & \text{if  } \sv(\sigma)<0
    \\
    s-1 & \text{if  } \sv(\sigma)>0.
  \end{cases}
\end{align*}
For a periodic point $q$, we define $\indx_p(q)=\indx(q)=\dim E^s_q$, which is
well-defined since the critical element $\gamma=\cO_G(q)$ must be hyperbolic for
a star flow.

We say $\sigma$ is {\em Lorenz-like}, if $\sv(\sigma)\neq0$ and
\begin{description}
\item[If $\sv(\sigma)>0$] then $\lambda_{s-1}<\lambda_s$, and
  $W^{ss}(\sigma)\cap C(\sigma)=\{\sigma\}$. Here $W^{ss}(\sigma)$ is the
  invariant manifold corresponding to the bundle $E^{ss}_\sigma$ of the
  partially hyperbolic splitting $T_\sigma M=E^{ss}_\sigma\oplus E^{cu}_\sigma$,
  where $E^{ss}_\sigma$ is the invariant space corresponding to the Lyapunov
  exponents $\lambda_1, \lambda_2, \cdots, \lambda_{s-1}$ and $E^{cu}_\sigma$
  corresponding to the Lyapunov exponents
  $\lambda_s, \lambda_{s+1}, \cdots, \lambda_{d}$.

\item[If $\sv(\sigma)<0$] then $\lambda_{s+1}<\lambda_{s+2}$, and
  $W^{uu}(\sigma)\cap C(\sigma)=\{\sigma\}$. Here $W^{uu}(\sigma)$ is the
  invariant manifold corresponding to the bundle $E^{uu}_\sigma$ of the
  partially hyperbolic splitting $T_\sigma M=E^{cs}_\sigma\oplus E^{uu}_\sigma$,
  where $E^{cs}_\sigma$ is the invariant space corresponding to the Lyapunov
  exponents $\lambda_1, \lambda_2, \cdots, \lambda_{s+1}$ and $E^{uu}_\sigma$
  corresponding to the Lyapunov exponents
  $\lambda_{s+2}, \lambda_{s+3}, \cdots, \lambda_{d}$.
\end{description}

The following shows that for star vector fields singularities in
nontrivial chain recurrent classes are Lorenz-like.

\begin{theorem}\label{thm:chainrecurrent1}
  For any $G\in\X^*(M)$ and $\sigma\in {\sing}(G)$, if the chain recurrent class
  $C(\sigma)$ is non-trivial, then any $\rho\in C(\sigma)$ is Lorenz-like.

  Moreover, there is a dense $\G_\delta$ subset $\G_1\subset\X^*(M)$
  such that, if we further assume that $G\in\G_1$, then all
  singularities in $C(\sigma)$ have the same periodic index
  $\indx_p(\rho)=\indx_p(\sigma)$.
\end{theorem}

\begin{proof}
  This is obtained in~\cite[Theorem 3.6]{SGW14}.
\end{proof}

Next result ensures that, generically among star vector fields, chain
recurrent classes are locally homogeneous.

\begin{theorem}\label{thm:Loc-Homoge}
  For a $C^1$ generic star vector field $G$ and any chain recurrent
  class $C$ of $G$, there is a neighborhood $U$ of $C$ in $M$ whose
  all the critical elements share the same periodic index with the
  critical elements within $C$.
\end{theorem}

\begin{proof}
  This is deduced in~\cite[Theorem 5.7]{SGW14}.
\end{proof}

Now we combine the previous results with Theorem~\ref{thm:star} and
Theorem~\ref{thm:attracting-Lorenz-hyp-sing} to prove
Corollary~\ref{mcor:starattractsechyp}.

\begin{proof}[Proof of Corollary~\ref{mcor:starattractsechyp}]
  Let $G\in\X^1(M)$ be robustly expansive admitting a transitive attractor
  with a trapping region $U_0\subset M$ on a $C^1$ neighborhood
  $\U\subset\X^1(M)$ of $G$. Then each $Y\in\U$ is a star vector field in $U$ by
  Theorem~\ref{thm:star} and $C(y)$, for each $y\in\Lambda_Y(U_0)$, equals
  $\Lambda_Y(U_0)=\omega(x(Y))$ for some $x(Y)\in U$ with dense forward orbit,
  and so $C(y)$ is nontrivial.

  Note that we arrive at this same conclusion if we start with a
  robustly transitive attractor $\Lambda_G(U_0)$ of a \emph{star}
  vector field $G$.

  If $\Lambda\cap\sing(G)=\emptyset$, then $\Lambda$ is hyperbolic
  (and so sectional-hyperbolic) since $G$ is a non-singular star
  vector field in $U$, by~\cite{GW2006}.

  Otherwise, every $\sigma\in\Lambda\cap\sing(G)$ is Lorenz-like, by
  Theorem~\ref{thm:chainrecurrent1}. Moreover, since $W^u_\sigma\subset\Lambda$,
  then every equilibria $\rho$ in $\Lambda$ must satisfy $\sv(\sigma)>0$. In
  addition, this property persist for all equilibria in $U$ for all $Y\in\U$ by
  the star property and, since $C(\sigma_Y)=\Lambda_Y(U_0)$ is non-trivial, we
  conclude that the periodic indices of all critical elements of $Y$ in $U$
  coincide, because we can choose $Y\in\G_1\cap\U$ from
  Theorems~\ref{thm:chainrecurrent1} and~\ref{thm:Loc-Homoge}.

  Hence we have hypothesis $(H1)$ and $(H2)$ of
  Theorem~\ref{thm:attracting-Lorenz-hyp-sing} for some $0<i+1<\dim M$
  and then $\Lambda$ is a sectional-hyperbolic set. The proof is
  complete.
\end{proof}


  \subsection{Robust chaotic attracting sets and
    sectional-hyperbolicity for $3$-flows}
\label{sec:robust-chaoticity}

We provide now proofs of Corollaries~\ref{mcor:sing-hyp-chaotic},
\ref{mcor:dcu2robchaotic} and~\ref{mcor:robust-chaotic-sing-hyp}

\begin{proof}[Proof of Corollary~\ref{mcor:sing-hyp-chaotic}]
  The assumption of sectional-hyperbolicity on an isolated proper
  subset $\Lambda$ with isolating neighborhood $U$ ensures that the
  maximal invariant subsets
  $\Lambda_Y(U)=\cap_{t>0}\overline{ \phi_(U)}$ for all $C^1$ nearby
  vector fields $Y$ are also sectional-hyperbolic attracting
  sets. Therefore, to deduce robust chaotic behavior in this setting
  it is enough to show that $\Lambda_Y(U)$ is chaotic with the same
  constant as $\Lambda$.

  Let $\Lambda$ be a sectional-hyperbolic attracting set for a $C^1$
  vector field $G$. Then there exists a strong-stable manifold
  $\cW^s_x$ through each $x\in U$ and we choose an adapted family of
  cross-sections $\Xi$ satisfying all the properties explained in
  Section~\ref{sec:preliminary-results}. Moreover, we can find a pair
  $\epsilon_0,\delta_0$ satisfying
  Theorem~\ref{thm:expansivepoincare}.

  We claim that $\Lambda$ is past chaotic with constant
  $r_0=\delta_0$. Indeed, arguing by contradiction, let us assume that
  there exists a neighborhood $U$ of $x$ so that
  $d(\phi_{-t}U, \phi_{-t}x) \le r_0$ for all $t>0$. Then we can find
  $y\in\cW^s_x\cap U$ such that $y\neq x$ and
  $d\big(\phi_{-t}y, \phi_{-t}x\big) \le r_0$ for every $t>0$. Since
  $\cW^s_x$ is uniformly contracted by the flow in positive time,
  there exists $\lambda>0$ such that
  $ d(y,x)\le Const\cdot e^{-\lambda t} d\big(\phi_{-t}y,
  \phi_{-t}x\big) \le Const\cdot\epsilon e^{-\lambda t} $ for all
  $t>0$. This contradicts the choice of $y\neq x$ and proves the
  claim.

  To obtain future chaotic behavior, we again argue by contradiction: we assume
  that $\Lambda$ is \emph{not future chaotic}. Then, for any given $\delta>0$,
  we can find a point $x\in\Lambda$ and an open neighborhood $V$ of $x$ such
  that the future orbit of each $y\in V$ is $\delta$-close to the future orbit
  of $x$, that is, $d\big(\phi_t y, \phi_t x\big)\le\delta$ for all $t>0$.

  We assume without loss of generality that we have chosen $\delta>0$ smaller
  than:
  \begin{itemize}
  \item half the size of the local stable leaves of points of the
    attracting set, and
  \item the size of the local unstable manifolds of the possible
    equilibria of $\Lambda$, and
    
  \item the value $\delta_0$ given by
    Theorem~\ref{thm:expansivepoincare} applied to $\Lambda$.
  \end{itemize}
  First, $x$ is not an equilibrium, for $y\in V\cap\cW^u_x$ would be
  sent $\delta$-away from $x$ for some $t>0$. Likewise, $x$ cannot
  be in the stable manifold $\cW^s_\sigma$ of a singularity
  $\sigma\in\Lambda$. For otherwise we can take a transversal disk $D$
  to $\cW^s_\sigma$ through $x$ contained in $V$, and use the
  Inclination Lemma (or $\lambda$-Lemma) to conclude that for any given
  $0<\xi<\delta$ and $T>0$ we can find $t>T$ and a neighborhood
  $W\subset D$ of $x$ such that $\phi_tW$ is $\xi$-$C^1$-close to
  $\cW^u_\sigma$ and $\phi_tx$ is $\xi$-close to $\sigma$. In
  particular, there exists $y\in W$ such that
  \begin{align*}
    d(\phi_t y,\phi_t x)
    \ge
    d(\phi_ty,\sigma) - d(\sigma,\phi_tx)
    \ge
    2\delta-\xi >  \delta.
  \end{align*}
  Therefore, $\omega(x)$ contains some regular point $z$ and we can
  take $\Sigma\in\Xi$ a transversal section to the vector field which
  is crossed by the positive trajectory of $z$.

  Hence, there are infinitely many times $t_n\nearrow\infty$ such that
  $x_n:=\phi_{t_n}x\in\Sigma$ and $x_n\to z$ when
  $n\nearrow\infty$. The assumption on $V$ ensures that each $y\in V$
  admits also an infinite sequence $t_n(y)\nearrow\infty$ satisfying
  \begin{align*}
    y_n:=\phi_{t_n(y)}(y)\in\Sigma \quad\text{and}\quad
    d(y_n, x_n)\le \delta.
  \end{align*}
  We can assume without loss of generality that $y\in V$ does not
  belong to $\cup_{t\in\RR}\cW^s_{\phi_t x}$, since this is a $C^1$
  immersed submanifold of $M$. Now we consider $\cW^s_{x_n}(\Sigma)$
  and $W^s_{y_n}(\Sigma)$.

  We have reproduced the setting Theorem~\ref{thm:expansivepoincare}
  with $h$ the identity, and so we must have $y\in\cW^s_{\phi_\eta x}$
  for some $\eta>0$, which contradicts the choice of $y$.  Hence,
  $\Lambda$ is future chaotic with constant $r_0=\delta_0$, and
  concludes the proof.
\end{proof}

To prove Corollary~\ref{mcor:dcu2robchaotic} we need the following technical
result.

\begin{proposition}\label{pr:simplespec} Let us fix $G\in\X^1(M)$ admitting an
  isolated set $\Lambda=\Lambda_G(U_0)$ and $\U\subset\X^1(M)$ be a
  $C^1$-neighborhood of $G$.
  \begin{enumerate}
  \item If a singularity $\sigma$ of $G$ in $\Lambda$ is not
    hyperbolic, then there exists a vector field $Y\in\U$ for which
    $\sigma$ is a non-hyperbolic equilibrium in $\Lambda_Y(U_0)$ such
    that $\spec(DY(\sigma))\cap i\RR =\{\pm i\omega\}$ for some
    $\omega\in\RR$ and
    \begin{enumerate}
    \item either $\omega=0$ and the corresponding eigenspace $E_0$ is
      one-dimensional;
    \item or $\omega\neq0$ and the corresponding eigenspace $E_\omega$
      is two-dimensional.
    \end{enumerate}

  \item If a periodic orbit $\gamma$ of $G$ in $\Lambda$ is not
    hyperbolic, then there exists a vector field $Y\in\U$ for which
    $\gamma$ is non-hyperbolic periodic orbit in $\Lambda_Y(U_0)$
    whose Poincar\'e first return map $f_Y$ to the cross-section
    $\Sigma=\exp_p\big(B(0,r)\cap N_p\big)$ through $p\in\gamma$, for
    some $0<r<r_0$, satisfies
    \begin{enumerate}
    \item $\spec(Df(p))\cap\sS^1=\{\lambda,\bar\lambda\}$ for some
      $\lambda\in\CC$;
    \item if $\lambda\in\RR$ ($\lambda=\pm1$), then the
      corresponding generalized eigenspace $E_\lambda$ is
      one-dimensional;
    \item if $\lambda\in\sS^1\setminus\RR$, then 
      the corresponding generalized eigenspace $E_\lambda$ is
      two-dimensional.
    \end{enumerate}
  \item In either case, for any $\xi>0$ there exists
    $Z\in\U$ and $p,q\in U_0$ so that, if $\phi_t$ is the
    flow of $Z$, then there exists $h\in S(\RR)$ such that
    $d\big(\phi_t p, \phi_{h(t)}q)<\xi$ for all $t\in\RR$
    but the orbits $\cO_Z(p)$ and $\cO_Z(q)$ are
    distinct. In particular, the result holds with $h=Id$ in
    the singular case.
  \item Moreover, in the periodic case, for any $\xi>0$ we can find
    $Z\in\U$ and $p,q\in\Sigma\cap U_0$ hyperbolic periodic points for
    $Z$ such that $d(p,q)<\xi$ and whose indices satisfy
    $\indx(\cO_G(p))=\indx(\cO_Z(p))=\indx(\cO_Z(q))-1$.
  \end{enumerate}
\end{proposition}

\begin{proof}
  This is a simple adaptation of the proof of \cite[Theorem
  4.3]{ArbSenSod12}.
\end{proof}

We can now prove Corollary~\ref{mcor:dcu2robchaotic} as a
application of Theorem~\ref{thm:attracting-Lorenz-hyp-sing} and
Corollary~\ref{mcor:sing-hyp-chaotic}.

\begin{proof}[Proof of Corollary~\ref{mcor:dcu2robchaotic}]
  It is enough to assume that $\Lambda=\Lambda_G(U)$ is a robustly
  chaotic partially hyperbolic attracting set with $d_{cu}=2$ on the
  trapping region $U$ for a $C^1$ vector field $G$, and show that
  $\Lambda$ must be sectional-hyperbolic.

  Robust chaoticity implies that
  in $\X^1(M)$ there are no sinks (otherwise it would contradict
  future chaoticity) nor sources (otherwise it would contradict
  past chaoticity) in $U$ with respect to each vector field $Y\in\V$. This
  argument prevents the existence of either periodic attracting or
  repelling orbits, or attracting or repelling equilibria.

  Since all critical elements have index
  $\ge d_s = \dim M-d_{cu}=\dim M-2$, all equilibria $\sigma\in U_0$
  must be hyperbolic of saddle-type with index $d_s$ or $d_s+1$. For
  otherwise, either $\sigma$ is hyperbolic with index $\dim M$, a
  sink; or $\sigma$ fails to be hyperbolic and we can use item (3) of
  Proposition~\ref{pr:simplespec} to obtain a pair of arbitrarily
  close equilibria whose orbits forever remain close for a vector
  field also arbitrarily $C^1$ close to $G$, contradicting robust
  chaoticity.

  Analogously, all periodic orbits in $U$ for $Y\in\V$ are hyperbolic
  of saddle-type with index $d_s$. For otherwise, either we have a
  hyperbolic periodic orbit of $Y$ with index $d_s+1$, a sink; or we
  have a non-hyperbolic periodic orbit with index $d_s$. Hence, by
  arbitrarily small $C^1$ perturbations of the vector field, we would
  find, through item (4) of Proposition~\ref{pr:simplespec}, a
  hyperbolic periodic orbit again with index $d_s+1$. This contradicts
  the $C^1$ robust chaotic assumption.

  Altogether, we have shown that $G$ satisfies hypothesis (H1) and
  (H2) of Theorem~\ref{thm:attracting-Lorenz-hyp-sing} with
  $i=d_s$. The conclusion of
  Theorem~\ref{thm:attracting-Lorenz-hyp-sing} completes the proof of
  the corollary.
\end{proof}

The proof of Corollary~\ref{mcor:robust-chaotic-sing-hyp} is
analogous.

\begin{proof}[Proof of Corollary~\ref{mcor:robust-chaotic-sing-hyp}]
  Let $\Lambda=\Lambda_G(U)$ be a robustly chaotic attracting set on
  the trapping region $U$ for a $C^1$ vector field $G$ in a
  $3$-manifold $M^3$. As in the previous proof, there exists a $C^1$
  neighborhood $\V$ of $G$ in $\X^1(M^3)$ so that there are no sinks
  nor sources in $U$ with respect to each vector field $Y\in\V$.

  Since the ambient manifold is three-dimensional, all equilibria must
  be hyperbolic of saddle-type with index $1$ or $2$.  For otherwise,
  either we have a hyperbolic fixed sink (index $3$) or source (index
  $0$); or a non-hyperbolic equilibria. Then thorugh item (3) of
  Proposition~\ref{pr:simplespec}, after an arbitrarily small $C^1$
  perturbation, we obtain a pair of arbitrarily close equilibria whose
  orbits remain close at all times. This would contradict the robust
  chaotic assumption.

  Analogously, all periodic orbits in $U$ for $Y\in\V$ are hyperbolic
  of saddle-type with index $1$. For otherwise, either we have a
  periodic sink (index $2$) or a source (index $0$), or else a
  non-hyperbolic periodic orbit $\gamma$.  In the latter case
  $\indx(\gamma)=0$ or $1$ and  we can use item
  (4) of Proposiion~\ref{pr:simplespec} to obtain a hyperbolic periodic orbit
  arbitrarily close to $\gamma$ for a $C^1$ nearby vector field with
  index $0$ or $2$. This contradicts again the robust chaotic assumption. 

  We have shown that $G$ satisfies hypothesis (H1) and (H2) of
  Theorem~\ref{thm:attracting-Lorenz-hyp-sing} with $i=1$. The
  conclusion follows.
\end{proof}


\def\cprime{$'$}

\bibliographystyle{abbrv}


\begin{thebibliography}{10}

\bibitem{ABS77}
V.~S. Afraimovich, V.~V. Bykov, and L.~P. Shil{'}nikov.
\newblock {On the appearence and structure of the Lorenz attractor}.
\newblock {\em {Dokl. Acad. Sci. USSR}}, {234}:{336--339}, {1977}.

\bibitem{AP37}
A.~Andronov and L.~Pontryagin.
\newblock {Syst{\`e}mes grossiers}.
\newblock {\em {Dokl. Akad. Nauk. USSR}}, {14}:{247--251}, {1937}.

\bibitem{An67}
D.~V. Anosov.
\newblock {Geodesic flows on closed Riemannian manifolds of negative
  curvature}.
\newblock {\em {Proc. Steklov Math. Inst.}}, {90}:{1--235}, {1967}.

\bibitem{araujo_2021}
V.~Araujo.
\newblock Finitely many physical measures for sectional-hyperbolic attracting
  sets and statistical stability.
\newblock {\em Ergodic Theory and Dynamical Systems}, 41(9):2706--2733, 2021.

\bibitem{araJSP21}
V.~Araujo.
\newblock On the statistical stability of families of attracting sets and the
  contracting lorenz attractor.
\newblock {\em Journal of Statistical Physics}, 182:53--68, 2021.

\bibitem{ArMel17}
V.~Araujo and I.~Melbourne.
\newblock Existence and smoothness of the stable foliation for sectional
  hyperbolic attractors.
\newblock {\em Bulletin of the London Mathematical Society}, 49(2):351--367,
  2017.

\bibitem{ArMel18}
V.~Araujo and I.~Melbourne.
\newblock Mixing properties and statistical limit theorems for singular
  hyperbolic flows without a smooth stable foliation.
\newblock {\em Advances in Mathematics}, 349:212 -- 245, 2019.

\bibitem{AraPac2010}
V.~Araujo and M.~J. Pacifico.
\newblock {\em Three-dimensional flows}, volume~53 of {\em Ergebnisse der
  Mathematik und ihrer Grenzgebiete. 3. Folge. A Series of Modern Surveys in
  Mathematics [Results in Mathematics and Related Areas. 3rd Series. A Series
  of Modern Surveys in Mathematics]}.
\newblock Springer, Heidelberg, 2010.
\newblock With a foreword by Marcelo Viana.

\bibitem{APPV}
V.~Araujo, M.~J. Pacifico, E.~R. Pujals, and M.~Viana.
\newblock Singular-hyperbolic attractors are chaotic.
\newblock {\em Transactions of the A.M.S.}, 361:2431--2485, 2009.

\bibitem{ArSzTr}
V.~Araujo, A.~Souza, and E.~Trindade.
\newblock Upper large deviations bound for singular-hyperbolic attracting sets.
\newblock {\em Journal of Dynamics and Differential Equations}, 31(2):601--652,
  2019.

\bibitem{arbieto2010}
A.~Arbieto.
\newblock Sectional lyapunov exponents.
\newblock {\em Proc. of the Amercian Mathematical Society}, 138:3171--3178,
  2010.

\bibitem{AMS2010}
A.~Arbieto, C.~Morales, and L.~Senos.
\newblock On the sensitivity of sectional-anosov flows.
\newblock {\em Mathematische Zeitschrift}, 270(1):545--557, 2012.

\bibitem{ArbSenSod12}
A.~Arbieto, L.~Senos, and T.~Sodero.
\newblock The specification property for flows from the robust and generic
  viewpoint.
\newblock {\em Journal of Differential Equations}, 253(6):1893 -- 1909, 2012.

\bibitem{MCA98}
M.-C. Arnaud.
\newblock {Le {``}closing lemma{''} en topologie $C\sp 1$}.
\newblock {\em {M{\'e}m. Soc. Math. Fr. (N.S.)}}, {74}:{--120}, {1998}.

\bibitem{artigue_2014}
A.~Artigue.
\newblock Positive expansive flows.
\newblock {\em Topology and its Applications}, 165:121--132, 2014.

\bibitem{artigue_2016}
A.~Artigue.
\newblock Kinematic expansive flows.
\newblock {\em Ergodic Theory and Dynamical Systems}, 36(2):390--421, 2016.

\bibitem{BaBoPac21}
D.~{Barros}, C.~{Bonatti}, and M.~J. {Pacifico}.
\newblock {Up, down, two-sided Lorenz attractor, collisions, merging and
  switching}.
\newblock {\em arXiv e-prints}, page arXiv:2101.07391, Jan. 2021.

\bibitem{BesLeeWen}
M.~Bessa, M.~Lee, and X.~Wen.
\newblock Shadowing, expansiveness and specification for {$C^1$}-conservative
  systems.
\newblock {\em Acta Mathematica Scientia}, 35(3):583 -- 600, 2015.

\bibitem{BPV97}
C.~Bonatti, A.~Pumari{\~n}o, and M.~Viana.
\newblock {Lorenz attractors with arbitrary expanding dimension}.
\newblock {\em {C. R. Acad. Sci. Paris S{\'e}r. I Math.}},
  {325}({8}):{883--888}, {1997}.

\bibitem{Bowen72}
R.~Bowen.
\newblock Entropy-expansive maps.
\newblock {\em Transactions of the American Mathematical Society},
  164:323--331, Feb. 1972.

\bibitem{BoWa72}
R.~Bowen and P.~Walters.
\newblock {Expansive one-parameter flows}.
\newblock {\em {J. Differential Equations}}, {12}:{180--193}, {1972}.

\bibitem{daLuz18}
A.~{da Luz}.
\newblock {Starflows with singularities of different indices}.
\newblock {\em arXiv e-prints}, page arXiv:1806.09011, June 2018.

\bibitem{Do87}
C.~I. Doering.
\newblock {Persistently transitive vector fields on three-dimensional
  manifolds}.
\newblock In {\em {Procs. on Dynamical Systems and Bifurcation Theory}}, volume
  {160}, pages {59--89}. {Pitman}, {1987}.

\bibitem{GW2006}
S.~Gan and L.~Wen.
\newblock Nonsingular star flows satisfy {A}xiom {A} and the no-cycle
  condition.
\newblock {\em Invent. Math.}, 164(2):279--315, 2006.

\bibitem{Gu76}
J.~Guckenheimer.
\newblock {A strange, strange attractor}.
\newblock In {\em {The Hopf bifurcation theorem and its applications}}, pages
  {368--381}. {Springer Verlag}, {1976}.

\bibitem{GW79}
J.~Guckenheimer and R.~F. Williams.
\newblock {Structural stability of Lorenz attractors}.
\newblock {\em {Publ. Math. IHES}}, {50}:{59--72}, {1979}.

\bibitem{He76}
M.~H{\'e}non.
\newblock {A two dimensional mapping with a strange attractor}.
\newblock {\em {Comm. Math. Phys.}}, {50}:{69--77}, {1976}.

\bibitem{KS79}
H.~B. Keynes and M.~Sears.
\newblock {F-expansive transformation groups}.
\newblock {\em {General Topology Appl.}}, {10}({1}):{67--85}, {1979}.

\bibitem{Km84}
M.~Komuro.
\newblock {Expansive properties of Lorenz attractors}.
\newblock In H.~Kawakami, editor, {\em {The theory of dynamical systems and its
  applications to nonlinear problems }}, pages {4--26}. World Scientific
  Publishing Co., Singapure, {1984}.
\newblock Papers from the meeting held at the Research Institute for
  Mathematical Sciences, Kyoto University, Kyoto, July 4--7, 1984.

\bibitem{Ko84}
S.~Kotani.
\newblock {Lyapunov indices determine absolutely continuous spectra of
  stationary random one-dimensional Schr{\"o}dinger operators}.
\newblock In {\em {Stochastic analysis}}, pages {225--248}. {North Holland},
  {1984}.

\bibitem{LP86}
R.~Labarca and M.~J. Pacifico.
\newblock {Stability of singular horseshoes}.
\newblock {\em {Topology}}, {25}:{337--352}, {1986}.

\bibitem{LY85}
F.~Ledrappier and L.~S. Young.
\newblock {The metric entropy of diffeomorphisms I. Characterization of
  measures satisfying Pesin{'}s entropy formula}.
\newblock {\em {Ann. of Math}}, {122}:{509--539}, {1985}.

\bibitem{LeplYa17}
R.~Leplaideur and D.~Yang.
\newblock {SRB} measure for higher dimensional singular partially hyperbolic
  attractors.
\newblock {\em Annales de l'Institut Fourier}, 67(2):2703--2717, 2017.

\bibitem{CerLew10}
J.~Lewowicz and M.~Cerminara.
\newblock Some open problems concerning expansive systems.
\newblock In {\em Rend. Instit. Mat. Univ. Trieste}, volume~42, pages 129--141.
  {EUT} Edizioni Universit{\`a} di Trieste, 2010.

\bibitem{Lo63}
E.~N. Lorenz.
\newblock {Deterministic nonperiodic flow}.
\newblock {\em {J. Atmosph. Sci.}}, {20}:{130--141}, {1963}.

\bibitem{Man88}
R.~Ma{\~n}{\'e}.
\newblock {A proof of the $C^1$ stability conjecture}.
\newblock {\em {Publ. Math. I.H.E.S.}}, {66}:{161--210}, {1988}.

\bibitem{MeMor08}
R.~Metzger and C.~Morales.
\newblock Sectional-hyperbolic systems.
\newblock {\em Ergodic Theory and Dynamical System}, 28:1587--1597, 2008.

\bibitem{MPP98}
C.~Morales, M.~J. Pacifico, and E.~Pujals.
\newblock {On $C^1$ robust singular transitive sets for three-dimensional
  flows}.
\newblock {\em {C. R. Acad. Sci. Paris}}, {326, S{\'e}rie I}:{81--86}, {1998}.

\bibitem{MPP00}
C.~Morales, M.~J. Pacifico, and E.~Pujals.
\newblock {Strange attractors across the boundary of hyperbolic systems}.
\newblock {\em {Comm. Math. Phys.}}, {211}({3}):{527--558}, {2000}.

\bibitem{MPu97}
C.~Morales and E.~Pujals.
\newblock {Singular strange attractors on the boundary of Morse-Smale systems}.
\newblock {\em {Ann. Sci. {\'E}cole Norm. Sup.}}, {30}:{693--717}, {1997}.

\bibitem{MPu98}
C.~Morales and E.~Pujals.
\newblock {Strange attractors containing a singularity with two positive
  multipliers}.
\newblock {\em {Comm. Math. Phys.}}, {196}:{671--679}, {1998}.

\bibitem{Morales07}
C.~A. Morales.
\newblock {Examples of singular-hyperbolic attracting sets}.
\newblock {\em {Dynamical Systems, An International Journal}},
  {22}({3}):{339{--}349}, {2007}.

\bibitem{MPP99}
C.~A. Morales, M.~J. Pacifico, and E.~R. Pujals.
\newblock Singular hyperbolic systems.
\newblock {\em Proc. Amer. Math. Soc.}, 127(11):3393--3401, 1999.

\bibitem{MPP04}
C.~A. Morales, M.~J. Pacifico, and E.~R. Pujals.
\newblock {Robust transitive singular sets for 3-flows are partially hyperbolic
  attractors or repellers}.
\newblock {\em {Ann. of Math. (2)}}, {160}({2}):{375--432}, {2004}.

\bibitem{MPS05}
C.~A. Morales, M.~J. Pacifico, and B.~{San Martin}.
\newblock {Expanding Lorenz attractors through resonant double homoclinic
  loops}.
\newblock {\em {SIAM J. Math. Anal.}}, {36}({6}):{1836--1861}, {2005}.

\bibitem{MorSakSun05}
K.~Moriyasu, K.~Sakai, and W.~Sun.
\newblock C1-stably expansive flows.
\newblock {\em Journal of Differential Equations}, 213(2):352 -- 367, 2005.

\bibitem{PaYaYa}
M.~J. Pacifico, F.~Yang, and J.~Yang.
\newblock Entropy theory for sectional hyperbolic flows.
\newblock {\em Annales de l'Institut Henri Poincar{\'e} C, Analyse non
  lin\'eaire}, 2020.

\bibitem{PM82}
J.~Palis and W.~{de Melo}.
\newblock {\em {Geometric Theory of Dynamical Systems}}.
\newblock {Springer Verlag}, {1982}.

\bibitem{Pugh67}
C.~C. Pugh.
\newblock {The closing lemma}.
\newblock {\em {Amer. J. Math.}}, {89}:{956--1009}, {1967}.

\bibitem{teseSenos}
L.~Senos.
\newblock {\em Expansividade, Especifica\c{c}\~ao e Sensibilidade em Fluxos com
  Singularidades}.
\newblock PhD thesis, Universidade Federal do Rio de Janeiro, 2010.

\bibitem{SGW14}
Y.~Shi, S.~Gan, and L.~Wen.
\newblock On the singular-hyperbolicity of star flows.
\newblock {\em J. Mod. Dyn.}, 8(2):191--219, 2014.

\bibitem{Sm67}
S.~Smale.
\newblock {Differentiable dynamical systems}.
\newblock {\em {Bull. Am. Math. Soc.}}, {73}:{747--817}, {1967}.

\bibitem{Tu99}
W.~Tucker.
\newblock {The Lorenz attractor exists}.
\newblock {\em {C. R. Acad. Sci. Paris}}, {328, S{\'e}rie I}:{1197--1202},
  {1999}.

\bibitem{ST98}
D.~V. Turaev and L.~P. Shil{'}nikov.
\newblock {An example of a wild strange attractor}.
\newblock {\em {Mat. Sb.}}, {189}({2}):{137--160}, {1998}.

\bibitem{Wen2018}
X.~Wen and L.~Wen.
\newblock A rescaled expansiveness for flows.
\newblock {\em Transactions of the American Mathematical Society},
  371(5):3179--3207, Nov. 2018.

\bibitem{Wil79}
R.~F. Williams.
\newblock {The structure of Lorenz attractors}.
\newblock {\em {Inst. Hautes {\'E}tudes Sci. Publ. Math.}}, {50}:{73--99},
  {1979}.

\end{thebibliography}

\end{document}